\documentclass[headsepline]{scrartcl} 

\usepackage{humanist}

\usepackage{scrpage2}
\clearscrheadfoot
\ohead{\pagemark}
\ihead{\headmark}

\usepackage[UKenglish]{babel}

\usepackage{amssymb, amsfonts, amsthm}
\usepackage{mathtools} 
\allowdisplaybreaks
\usepackage[thicklines]{cancel}

\usepackage{dsfont}
\usepackage{bbm}
\usepackage{upgreek}
\usepackage{bm}

\usepackage{capt-of} 
\usepackage{caption} 
\usepackage{longtable}
\usepackage{multirow}
\usepackage{colortbl}

\usepackage{enumerate}

\usepackage{multicol}

\usepackage{color}
\usepackage{graphicx}
\usepackage{tikz}
\tikzset{x tick label/.style={anchor=north, minimum width=7mm}}
\definecolor{gray70}{RGB}{179,179,179}
\definecolor{gray73}{RGB}{186,186,186}
\definecolor{darkgreen}{RGB}{0,100,0}
\definecolor{mygreen}{RGB}{34,139,34}

\usepackage{epstopdf}

\usepackage{marginnote}

\usepackage{currvita}

\usepackage{cite}

\usepackage[colorlinks, linkcolor=darkgreen, citecolor = darkgreen, filecolor = darkgreen, urlcolor = darkgreen]{hyperref}

\theoremstyle{plain}
\newtheorem{theorem}{Theorem}[section]    
    
\newtheorem{lemma}[theorem]{Lemma}    

\theoremstyle{definition}

\newtheorem{corollary}[theorem]{Corollary}  
\newtheorem{remark}[theorem]{Remark}
\newtheorem{remarks}[theorem]{Remarks}

\renewenvironment{proof}{\noindent\textbf{Proof. }\ignorespaces}{\unskip\hspace*{\fill}\(\Box\)}

\newcommand{\N}{\mathbb{N}}
\newcommand{\Z}{\mathbb{Z}}
\newcommand{\C}{\mathbb{C}}
\newcommand{\R}{\mathbb{R}}
\newcommand{\D}{\mathbb{D}}
\newcommand{\G}{{\mathbf \Gamma}}

\newcommand{\Ker}{\textnormal{ker}}

\newcommand{\sh}{{\mathbf S}}

\renewcommand{\pi}{\uppi}
\renewcommand{\tau}{\uptau}
\renewcommand{\delta}{\updelta}
\renewcommand{\Phi}{\varPhi}
\renewcommand{\Lambda}{\varLambda}
\DeclareMathOperator{\rank}{rank}

\makeatletter
\newif\ifkp@upRm
\DeclareSymbolFont{Letters}{OML}{jkp}{m}{it}
\DeclareMathSymbol{\partialup}{\mathord}{Letters}{128}
\makeatother
\renewcommand{\partial}{\partialup}
\renewcommand{\text}{\textup}
\usepackage[left=3cm,right=2cm,top=1.5cm,bottom=1.5cm,includeheadfoot]{geometry}

\begin{document}

\hyphenation{Ma-the-ma-tik ma-the-ma-tischen ana-lysis multi-pli-ca-tion de-riva-tive de-riva-tives trans-form-ation ap-proxi-mate eu-clid-ean de-ter-mine rep-re-sen-ta-tion rep-re-sen-ta-tions posi-tive per-pen-dicu-lar math-emat-ics com-pari-son de-ter-mined exam-ine exam-ined ex-peri-ment ex-peri-ments de-ter-min-ant de-ter-min-ants exact-ly re-sult re-sults es-ti-ma-tion con-eigen-val-ues}

\pdfbookmark[1]{Title}{title}
\title{Application of the AAK theory for sparse approximation of exponential sums}
\author {Gerlind Plonka\footnote{Institute for Numerical and Applied Mathematics, G\"ottingen University, Lotzestr.\ 16-18,  37083 G\"ottingen, Germany. Email: plonka@math.uni-goettingen.de} \qquad Vlada Pototskaia\footnote{Institute for Numerical and Applied Mathematics, G\"ottingen University, Lotzestr.\ 16-18,  37083 G\"ottingen, Germany. Email: v.pototskaia@math.uni-goettingen.de}}

\maketitle

\pdfbookmark[1]{Abstract}{abstract}
\abstract{In this paper, we derive a new method for optimal $\ell^{1}$- and $\ell^2$-approximation of discrete signals on ${\N}_{0}$
whose entries can be represented as an exponential sum of finite length. 
Our approach employs Prony's method in a first step to recover the exponential sum that is determined by the signal. In the second step we use the AAK-theory to derive an algorithm for computing a shorter exponential sum that approximates the original signal in the $\ell^{p}$-norm well.
AAK-theory originally determines best approximations of bounded periodic functions in Hardy-subspaces. 
We rewrite these ideas for our purposes and give a proof of the used AAK theorem based only on basic tools from linear algebra and Fourier analysis. The new algorithm is tested numerically in different examples.
\medskip


}

\section{Introduction}
\setcounter{equation}{0}

In signal processing and system theory, we consider the problem of sparse approximation of structured signals.
In many applications it can be assumed that the signal is either exactly or approximately a finite linear combination of non-increasing exponentials with complex exponents, i.e., 
${\mathbf f}:= (f_{k})_{k=0}^{\infty}$ satisfies 
\begin{equation}\label{eq:expsum}
f_{k} := f(k) = \sum_{j=1}^N a_j\, z_{j}^{k}\,, 
\end{equation}
where $a_{j} \in {\mathbb C} \setminus \{ 0 \}$ and with pairwise different $z_{j} \in \mathbb D:= \{z\in {\mathbb C}:\, 0 < |z| < 1\}$.
The problem of recovering ${\mathbf f}$ from  a suitable number of signal values $f(k)$, $k =0, 1, \ldots , M$ with $M\ge 2N-1$ is 
a well-studied parameter estimation problem. It can be solved using  a stabilized Prony-like method as e.g.\ ESPRIT \cite{Roy} or the 
approximative Prony method (APM), \cite{APM}.  For a recent generalization and a review on Prony methods  for recovery of structured  functions we refer to  \cite{PP13,PT14}.
Beside system theory, the approximation of special functions (as e.g. $1/x$) by finite exponential sums is also of high interest for solving higher dimensional integral equations as e.g.\ in coupled cluster analysis in quantum chemistry, see \cite{BH05,HB}.
\smallskip

For  sparse signal approximation we are interested in solving the following problem.
For a given  ${\mathbf f}$ of the form (\ref{eq:expsum}), we want to find a
 new signal $ \tilde{\mathbf f}:= (\tilde{f}_{k})_{k=0}^{\infty}$ of the form
\begin{equation} 
\label{ff} \tilde{f}_{k} := \tilde{f}(k) = \sum_{j=1}^K \tilde{a}_j\, \tilde{z}_{j}^{k} 
\end{equation}
with $\tilde{a}_{j} \in {\mathbb C}$ and $\tilde{z}_{j} \in \mathbb D$ such that $K < N$ and
 $\| {\mathbf f}-\tilde{{\mathbf  f}} \|_{p} \le \epsilon$.
We will focus on the two cases $p=1$ and $p=2$.
 \\
 
More precisely, we have to consider the following two questions. For a given accuracy level $\epsilon >0$, what is the smallest $K \in {\N}$
such that $\tilde{\mathbf f}$ in (\ref{ff}) satisfies $\| {\mathbf f}- \tilde{{\mathbf  f}} \|_{{p}} \le \epsilon$, and how to compute $\tilde{z}_{j} \in {\D}$ and $\tilde{a}_{j} \in \C$, $j=1, \ldots , K$?
Vice versa, for a given ``storage budget'', i.e., a given $K \in {\N}$, how  do we have to choose the parameters in (\ref{ff}) in order to achieve the smallest possible error $\| {\mathbf f}-  \tilde{{\mathbf f}} \|_{{p}}$?
\medskip

In this paper, we will approach this problem using the theory of Adamjan, Arov and Krein (AAK-theory) \cite{AAK}.
Consider the infinite Hankel matrix generated by ${\mathbf f}= (f_{k})_{k=0}^{\infty}$ of the form $\G_{\mathbf f} = (f_{j+k})_{j,k=0}^{\infty}$ being bounded on $\ell^{p}$, and let $\sigma_{0} \ge \sigma_{1} \ge \ldots $ denote its singular values, ordered by size and repeated according to multiplicities. Then the AAK-theory states that $\G_{\mathbf f}$ can be approximated by an infinite Hankel matrix $\G_{\tilde{\mathbf f}}$ of finite  $K$ such that 
$$ \| \G_{\mathbf f} - \G_{\tilde{\mathbf f}} \|_{\ell^{p} \to \ell^{p}} = \min_{{\rank}\G_{\mathbf g} \le K} \| \G_{\mathbf f} - \G_{{\mathbf g}} \|_{\ell^{p} \to \ell^{p}} = \sigma_{K}. $$
This result ist non-trival since a usual spectral decomposition of $\G_{\mathbf f}$ does not preserve the Hankel structure. Thus the result is strongly related to the problem of structured low rank approximation for Hankel matrices, see e.g. \cite{Mark}.
As presented in \cite{AAK}, the problem is equivalent to showing that  for an arbitrary  bounded function $f \in L^{\infty}([0, 2\pi))$, the best approximation by a function $\tilde{f}$ from the subspace $H^{\infty, [K]}$ $(K \in {\N})$ of the Hardy space $H^{\infty}$ exists and satisfies 
$$ \| f - \tilde{f} \|_{\infty}  = \min_{g \in H^{\infty, [K]}} \| f - g 
\|_{\infty} = \sigma_{K}(\G_{\mathbf f}), $$
where $\G_{\mathbf f}$ is generated by ${\mathbf f} = (f(k))_{k=0}^{\infty}$. Here $H^{\infty, [K]}$ denotes the space of functions in $H^{\infty}$ whose
extension to the the unit disc possess at most $K$ poles in $\D$. For more details we refer e.g.\ to \cite{AAK, Meinguet, Young, Chui, Nikol, Peller, Martinez}.
The proofs of these assertions  are substantially based on analysis of bounded functions in Hardy spaces  and operator theory, employing the 
Nehari's Theorem \cite{Nehari}, Kronecker's Theorem (see e.g. \cite{Young}, Theorem 16.3) as well as Beurling's Theorem \cite{Beurling, Chui}.
\medskip

For earlier approaches  to the application of the AAK theory in order to solve sparse approximation  problems using exponential sums we refer to 
 \cite{BM05} and \cite{anderson}. 
In \cite{BM05}, a finite-dimensional  approximation problem  is considered using $2N+1$ equidistant samples of a continuous function $f$. The goal is to approximate $f$ by an exponential sum as accurate as possible with $K < N$ terms. Using $N \times N$ Hankel matrices $H_{N}$, the authors derived 
a reduction procedure such that  the spectral norm of the difference  between $H_{N}$ and   the obtained rank $K$ Hankel approximation  $\tilde{H}_{N}$  is approximately $\sigma_{K}$, where $\sigma_{K}$ is the $K$-th singular value of $H_{N}$. 
In \cite{anderson}, relations between the AAK-theory and related discrete and continuous approximation problems on ${\R}^{+}$ and on the interval have been studied. However, the connection between AAK-theory  and corresponding  finite-dimensional approximation problems is still not completely understood.
\medskip

{\bf Contribution of our paper.}
In this paper, we present a new algorithm for solving the sparse approximation problem (\ref{ff}) and give an explicit procedure to compute the nodes $\tilde{z}_{j}$ as well as the coefficients $\tilde{a}_{j}$, $j=1, \ldots, K$ such that $\| {\mathbf f} - \tilde{\mathbf f} \|_{{p}}  \le \sigma_{K}$
is satisfied, where $\sigma_{K}$ denotes the $K$-th singular value of $\G_{\mathbf f}$. The procedure also includes an explicit computation of all non-zero singular values of $\G_{\mathbf f}$.
For this purpose, we investigate the structure of con-eigenvectors of the infinite Hankel matrix $\G_{\mathbf f}$  with finite rank $N$  and reduce the problem of characterizing the nonzero  singular values  of $\G_{\mathbf f}$  to the problem of finding the singular values of an $N \times N$
kernel matrix.
Further, we provide a new proof of the AAK Theorem in our context using only concepts from linear algebra and Fourier analysis.
We show how these results intimately relate to the Prony method for recovering exponential sums.

The numerical application of our approach usually employs in the first step a stabilized Prony method, like APM \cite{APM} to recover the parameters of ${\mathbf f}$ in (\ref{eq:expsum}). In a second step, the new reduction method is used  to compute all parameters of $\tilde{\mathbf f}$. For a fixed target error $\epsilon >0$ we choose $K$ as the smallest index  such that $\sigma_{K} \le \epsilon$. 
\medskip

This paper is organized as follows.
In Section 2 we summarize some basic notations and properties of infinite Hankel matrices.
Section 3 is devoted to the derivation of the new algorithm for sparse approximation by short exponential sums.
In Section 4.1 we provide some important properties of infinite Toeplitz and Hankel matrices.
The assertions of the AAK theory, stated already in Section 3, are proven in Subsection 4.2. Moreover, we give more insights
into the structure of singular values, (con)-eigenvectors as well as the kernel of finite rank Hankel matrices and its rank $K$ Hankel approximations. These insights provide the close relation of the theory to Prony's method.
Finally, in Section 5 we present some numerical examples  and applications of the new sparse approximation algorithm
and comment on stability issues.


\section{Preliminaries}
\label{preliminaries} 
\setcounter{equation}{0}

In the following we denote by $\ell^p:=\ell^p(\N_{0})$ for $p=1,2$ the space of $p$-summable sequences $ {\mathbf v}=(v_k)_{k=0}^\infty$ with the norm $\|v\|_p:=\left(\sum_{k=0}^\infty |v_p|^p\right)^{1/p}$ and by $\D$ the open unit disc without zero $\{z\in\C:0<|z|<1\}$. For a sequence ${\mathbf  v} = (v_{k})_{k=0}^{\infty} \in \ell^{p}$ and $z \in {\D}$
we call 
$$ P_{\bf  v}(z) := \sum_{k=0}^{\infty} v_{k} z^{k}$$
its corresponding Laurent polynomial and $P_{\bf v}(e^{i \omega})$, $\omega \in {\R}$, its Fourier series. Further, for ${\bf f}\in\ell^1$
we define the infinite Hankel matrix
\begin{equation}\label{gamma}
	{\mathbf \Gamma}_{\mathbf f}:=\left(
	\begin{array}{cccc}
	f_0 & f_1 & f_2 & \cdots \\
	f_1 &  f_2 & f_3 & \cdots \\
	f_2 &  f_3 & f_4 & \cdots \\
	\vdots & \vdots & \vdots & \ddots \\
	\end{array}
	\right) = \left(f_{k+j}\right)_{k,j=0}^\infty.
\end{equation}

Then $\G_{\mathbf f}$ determines an operator $\G_{\mathbf f}:\ell^p \to \ell^p$ given by
$$
	\G_{\mathbf f} \, {\mathbf v}  = \left( \sum_{j=0}^\infty f_{k+j}v_j \right)_{k=0}^\infty \quad \textnormal{for} ~ {\mathbf v}:=(v_k)_{k=0}^\infty \in \ell^p.
$$
This is a direct consequence of Youngs inequality, since $\G_{\mathbf f} \, {\mathbf v}$ can be easily reinterpreted as a convolution by extending the sequence spaces to $\ell^{p}({\Z})$.

\noindent
{\bf Con-diagonalization of infinite Hankel matrices}. Observe that $\G_{\mathbf f}$ is symmetric. Generalizing the idea of unitary diagonalization of Hermitian matrices resp. compact selfadjoint operators, we will apply the concept of con-similarity and con-diagonalization, see e.g. \cite{HJ} for the finite-dimensional case.

For an infinite Hankel matrix $\G_{\mathbf f}$ we call $\lambda\in\mathbb{C}$ a {\it con-eigenvalue} with the corresponding {\it con-eigenvector} $ {\mathbf v} \in \ell^p$ if it satisfies
$$
\G_{\mathbf f} \overline{\mathbf v} = \lambda \, {\mathbf v}.
$$
Observe that for $\G_{\mathbf f} \overline{\mathbf v} = \lambda \, {\mathbf v}$ also
$$ \G_{\mathbf f} ( \overline{{\mathrm e}^{{\mathrm i} \alpha} {\mathbf v}}) = {\mathrm e}^{-{\mathrm i} \alpha } \G_{\mathbf f} \overline{\mathbf v} =  ({\mathrm e}^{-{\mathrm i} \alpha} \lambda) {\mathbf v} = ({\mathrm e}^{-2{\mathrm i} \alpha} \lambda) ({\mathrm e}^{{\mathrm i} \alpha} {\mathbf v}) $$
for all $\alpha \in {\R}$. Thus, for each con-eigenvalue  $\lambda$ of $\G_{\mathbf f}$  we can find a corresponding real non-negative con-eigenvalue $\sigma = |\lambda|$ by this rotation trick.
In the following, we will restrict  the con-eigenvalues to their unique nonnegative representatives.
Now, it can be simply observed that a symmetric infinite Hankel matrix  $\G_{\mathbf f}$ with finite rank is compact and unitarily con-diagonalizable, see \cite{HJ}. Since $\G_{\mathbf f} \overline{\mathbf v} = \lambda \, {\mathbf v}$ implies 
$$ ( \G_{\mathbf f}^{*} \G_{\mathbf f}) {\mathbf v} = \G_{\mathbf f}^{*} \lambda \overline{\mathbf v} = \lambda \overline{\G}_{\mathbf f} \overline{\mathbf v} = \lambda \overline{ \G_{\mathbf f} {\mathbf v}} = | \lambda|^{2} {\mathbf v},$$
we directly observe that the nonnegative con-eigenvalues and con-eigenvectors of $\G_{\mathbf f}$ are also singular values and corresponding  singular vectors of $\G_{\mathbf f}$. Conversely, for symmetric matrices a singular pair $(\sigma, {\mathbf v})$ is also a con-eigenpair of $\G_{\mathbf f}$ if the geometric multiplicity of $\sigma$ is $1$.\\


\noindent
{\bf Infinite Hankel matrices of finite rank.}
In the following, we consider special sequences ${\mathbf f}= (f_{k})_{k=0}^{\infty}$. 	
Let ${\textbf f}$ be of the form (\ref{eq:expsum})
where $N \in {\N}$, $a_{j} \in {\C}\neq \{ 0 \}$ and with pairwise different nodes $0 < |z_{N}| \le \ldots \le |z_{1}| < 1$. Then ${\mathbf f} \in \ell^{1}$ since
$$
 \| {\mathbf f} \|_{1}=  \sum_{k=0}^{\infty} | f_{k}| = \sum_{k=0}^{\infty} \left| \sum_{j=1}^{N} a_{j} z_{j}^{k} \right| 
 \le  \sum_{j=1}^{N} \left( \sum_{k=0}^{\infty} |a_{j} z_{j}^{k} | \right) = \sum_{j=1}^{N} \frac{|a_{j}|}{1-|z_{j}|} < \infty. 
$$
First, we recall the following property of the corresponding infinite Hankel matrix $\G_{\mathbf f}$, see e.g. \cite{Young}, Theorem 16.13.

\begin{theorem} \label{Kronecker} {\rm (Kronecker's Theorem).}
The Hankel operator $\G_{\mathbf f}: \ell^{p} \to \ell^{p}$ generated by ${\mathbf f} = (f_{k})_{k=0}^{\infty} \in \ell^{1}$  of the form {\rm (\ref{eq:expsum})} has finite rank $N$.
\end{theorem}

\begin{proof}
For reader's convenience we provide a short proof that also gives some insight into the connection to difference equations.
If ${\mathbf f}$ can be written in the form (\ref{eq:expsum}), we define the characteristic polynomial  (Prony polynomial)
\begin{equation} \label{prony} 
P(z):= \prod_{j=1}^{N} (z-z_{j}) = \sum_{k=0}^{N} b_{k} z^{k}. 
\end{equation}
Then 
\begin{equation}\label{dif} \sum_{l=0}^{N} b_{l} f_{k+l} = \sum_{l=0}^{N} b_{l} \sum_{j=1}^{N} a_{j} z_{j}^{k+l} = \sum_{j=1}^{N} a_{j} z_{j}^{k} (\sum_{l=0}^{N} p_{l} z_{j}{l}) =0
\end{equation}
for all $k \in \N_{0}$, i.e., the $(N+k)$-th column  of $\G_{\mathbf f}$ is a linear combination of the $N$ preceding columns. Thus $\rank~\G_{\mathbf f} \le N$.
Since $P(z)$ has exact degree $N$ it follows that $ \rank~\G_{\mathbf f} = N$.
\end{proof}

\begin{remark}
Conversely, if the infinite Hankel matrix  $\G_{\mathbf f}$ possesses rank $N$ then ${\mathbf f}$ satisfies a difference equation of order $N$.
Thus, there exist coefficients $b_{0}, \ldots , b_{N}$ such that 
$$ \sum_{l=0}^{N} b_{l} f_{k+l} =0 \quad \forall k \in \N_{0}. $$
Assuming that  the zeros $z_{j}$, $j=1, \ldots , N$ of the characteristic polynomial $\sum_{l=0}^{N} b_{l} z^{l}$ are pairwise different, ${\mathbf f}$ can  be written in the form (\ref{eq:expsum}). The zeros have modulus smaller than $1$ since ${\mathbf f}$ has been assumed to be in $\ell^{1}$.
\end{remark}


\section{Algorithm for sparse approximation of exponential sums}
\label{sec2} 
\setcounter{equation}{0}

Let us assume now that  the signal ${\mathbf f}=(f_{k})_{k=0}^{\infty}$ is of the form (\ref{eq:expsum}) with pairwise different nodes $z_{j}$. Further, let a 
suitable number of it's samples $f_{k}$, $k =0, 1, \ldots , M$ with $M\ge 2N-1$ be given. 
 Our approach for solving the sparse approximation problem (2) consists of two steps. 
 In the first step we reconstruct the parameters $a_{j}$ and $z_{j}$ using Prony's method, see e.g. \cite{Roy,APM,PT14}. Once the representation (\ref{eq:expsum}) of ${\mathbf f}$ is known, in the second step we apply the AAK-theory \cite{AAK}  to compute a new signal $\tilde{\mathbf f}$ that can be represented by shorter  exponential sum (\ref{ff}) with $K < N$ and satisfies $\| {\mathbf f} - \tilde{\mathbf f} \|_{{p}} < \epsilon$.
\medskip

\noindent
{\bf Classical Prony method.} Let us shortly summarize  Prony's method  for the recovery of exponential sums. As seen already  in the proof of Theorem \ref{Kronecker}, a signal ${\mathbf f}$ of the form (\ref{eq:expsum}) satisfies a homogeneous difference equation  (\ref{dif}) of order $N$ with constant coefficients $b_l$ being determined by the coefficients of the Prony polynomial $P(z)$ in (\ref{prony}). For given samples $f_{k}$, $k=0, \ldots, 2N$, (\ref{dif}) leads to the homogeneous system of equations 
$$ \sum_{l=0}^{N} b_{l} f_{l+k} = 0, \qquad k=0, \ldots , N, $$
whose coefficient matrix $H_{N+1} \in {\C}^{(N+1) \times (N+1)}$ is the leading principal minor of $\G_{\mathbf f}$ with $\rank ~H_{N+1} = N$ by Theorem \ref{Kronecker}. Recalling that the leading coefficient of the Prony polynomial satisfies $b_{N}=1$,
there exists a unique solution of the homogeneous equation system, namely the eigenvector ${\mathbf b} = (b_{l})_{l=0}^{N}$ of $H_{N+1}$ to the single eigenvalue $0$ normalized by $b_{N}=1$. This observation leads to the following naive algorithm to recover the parameters $z_{j}, \, a_{j}$ in (\ref{eq:expsum}) by the signal samples $f_{k}$, $k=0, \ldots , 2N$.
	\begin{enumerate}
		\item Compute the eigenvector ${\mathbf b}$ of $H_{N+1}$ with $b_{N}=1$ corresponding to the zero-eigenvalue.
		\item Determine the Prony polynomial $P(z)$ in (\ref{prony}) with coefficient vector ${\mathbf b}$ and compute its $N$ zeros $z_j,~j=1,\ldots,N$.
		\item Solve the overdetermined linear system (\ref{eq:expsum}) with $k=0, \ldots , 2N$ to determine the weights $a_j,~j=1,\ldots,N$.
	\end{enumerate}
	
\begin{remarks}
1. We observe that the zero-eigenvector of $H_{N+1}$ also produces a zero-(con)-eigenvector of the infinite Hankel matrix $\G_{\mathbf f}$ by taking ${\mathbf v}^{(N)} = (v_{l}^{(N)})_{l=0}^{\infty}$ with $v_{l}^{(N)}= b_{l}$ for $l=0, \ldots , N$ and $b_{l}=0$ for $l>N$. Then the Laurent polynomial corresponding to ${\mathbf v}^{(N)}$ given by
$P_{{\mathbf v}^{(N)}}(z):=\sum_{l=0}^{\infty} v_{l}^{(N)} z^{l}$ equals to the Prony polynomial $P(z)$ in (\ref{prony}).

\noindent
2. Numerically stable algorithms for Prony's method usually employ more than $2N$ samples to achieve stability and use e.g. singular value decompositions of (rectangular) Hankel matrices as well as matrix pencil methods for evaluating  the nodes $z_{j}$, $j=1, \ldots , N$. In our numerical tests, we particularly use the APM method for stable evaluation of the parameters, see \cite{APM}.

\end{remarks}	
\medskip	

\noindent
{\bf Sparse approximation based on  AAK-theory.} 
In order to compute an optimal approximation $\tilde{\mathbf f}$ of ${\mathbf f}$ using a shorter exponential sum, we want to apply the following theorem based on AAK theory, \cite{AAK}.

\begin{theorem}\label{AAK} 
Let the Hankel matrix ${\mathbf \Gamma}_{\mathbf f}$ of rank $N$ be generated by the the sequence ${\mathbf f}$ of the form {\rm (\ref{eq:expsum})}
with $1 > |z_{1}| \ge \cdots \ge |z_{N}| >0$. Let the $N$ nonzero singular values of ${\mathbf \Gamma}_{\mathbf f}$ be ordered by size  $\sigma_{0} \ge  \sigma_{1} \ldots \ge \sigma_{N-1} >0$.
Then, for each $K \in \{0, \ldots , N-1 \}$ satisfying $\sigma_K \neq \sigma_{k}$ for $K \neq k$  the Laurent polynomial of the corresponding con-eigenvector ${\mathbf v}^{(K)} = (v^{(K)}_{l})_{l=0}^{\infty}$,
$$
	P_{{\mathbf v}^{(K)}}(z) := \sum_{l=0}^\infty v^{(K)}_{l} \, z^l,
$$
has exactly ${K}$ zeros ${z}^{(K)}_1,\ldots,{z}^{(K)}_K$ in $\mathbb{D}$, repeated according to their multiplicity. 
Furthermore, if ${z}^{(K)}_1,\ldots,{z}^{(K)}_K$ are pairwise different, then there exist coefficients $\tilde{a}_1, \ldots, \tilde{a}_{K} \in \mathbb{C}$ such that for
\begin{equation}\label{eq:AAK1}
	\tilde{{\mathbf f}}^{(K)} = \left(\tilde{f}_l^{(K)} \right)_{l=0}^{\infty} = \left(\sum_{j=1}^{{K}} \tilde{a}_j (z^{(K)}_j)^l\right)_{l=0}^\infty
\end{equation}
we have for $p=1$ and $p=2$
$$
	\| {\mathbf f} - {\tilde{\mathbf f}^{(K)}} \|_{{p}} \le \sigma_{{K}}.
$$
\end{theorem}

We will give a proof of this theorem in Section 4, where we also provide more insights into the structure of infinite Hankel matrices with finite rank and its spectral properties. 

%

In order to apply this theorem to the sparse approximation problem (\ref{ff}) we need to find a numerical procedure to compute the singular pairs $(\sigma_{{n}}, {\bf v}^{(n)})$ of ${\bf \Gamma}_{\bf f}$ for $n=0, \ldots , N-1$ and to find all zeros of the expansion $P_{{\bf v}^{(n)}}(z)$ lying inside $\mathbb D$. In a final step we have to compute the optimal coefficients $\tilde{a}_j$. 
\smallskip



Investigating the special structure of the con-eigenvectors of $\G_{\mathbf f}$ corresponding to the non-zero con-eigenvalues (resp.\ singular values) we show the following result that provides us with an algorithm to compute all non-zero singular values of $\G_{\mathbf f}$ and the corresponding con-eigenvectors exactly.

\begin{theorem} \label{reduction}
Let ${ {\mathbf f}}$ be of the form {\rm (\ref{eq:expsum})}. Then the con-eigenvector ${\mathbf v}^{(l)} = (v_{k}^{(l)})_{k=0}^{\infty}$ of $\G_{\mathbf f}$ corresponding to a single nonzero singular value $\sigma_{l}$ of ${\bf \Gamma}_{\bf f}$, $l \in \{0, 1, \ldots , N-1\}$  is of the form 
\begin{equation}\label{eq:coneigenred}
  v_{k}^{(l)} = \frac{1}{\sigma_{l}} \sum_{j=1}^{N} a_{j} P_{\overline{\mathbf v}^{(l)}} (z_{j}) z_{j}^{k}, \qquad k\in {\mathbb N}_{0},
\end{equation}
where the vector $(P_{\overline{\mathbf v}^{(l)}} (z_{j}))_{j=1}^{N} = (\overline{P_{{\mathbf v}^{(l)}} (\overline{z}_{j})})_{j=1}^{N}$ is determined by the con-eigenvector of the finite con-eigenvalue problem
\begin{equation} \label{eig}
  {\mathbf A}_{N} {\mathbf Z}_{N} (\overline{P_{{\mathbf v}^{(l)}} (\overline{z}_{j})})_{j=1}^{N}
 = \sigma_{l} (P_{{\mathbf v}^{(l)}} (\overline{z}_{j}))_{j=1}^{N}\end{equation}
with 
$$ {\mathbf A}_{N} := \left(\begin{array}{cccc} a_1 & & & 0 \\ & a_2 & & \\ & & \ddots & \\ 0 & & & a_N \end{array}\right), \quad
	{\mathbf Z}_{N}:=\left(\begin{array}{cccc} 	\frac{1}{1-|z_1|^2} 		& \frac{1}{1-z_1\bar{z}_2} 	& \cdots 	& \frac{1}{1-z_1\bar{z}_N} \\
					
						\frac{1}{1-\bar{z}_1z_2} 	& \frac{1}{1-|z_2|^2} 		& \cdots	& \frac{1}{1-z_2\bar{z}_N}	\\
						\vdots				& \vdots				& \ddots 	& \vdots				\\
						\frac{1}{1-\bar{z}_1z_N}	& \frac{1}{1-\bar{z}_2z_N}	& \cdots	& \frac{1}{1-|z_N|^2}		\\
		\end{array}\right).
$$
\end{theorem}

\begin{proof}
Let $(\sigma_{l}, {\mathbf v}^{(l)})$ with $\sigma_{l} \neq 0$ be a con-eigenpair of $\G_{\mathbf f}$, i.e. $\G_{\mathbf f} \overline{\mathbf v}^{(l)} = \sigma \, {\mathbf v}^{(l)}$. With the notation 
 $P_{\overline{\mathbf v}^{(l)}}(z) :=
\sum_{k=0}^{\infty} \overline{v}^{(l)}_{k} z^{k}$ it follows by (\ref{eq:expsum}) that 
$$ \sigma_{l} v_{k}^{(l)} = (\G_{\mathbf f} \overline{\mathbf v}^{(l)})_{k} = \sum_{r=0}^{\infty} f_{k+r} \overline{v}^{(l)}_{r}
= \sum_{r=0}^{\infty} \sum_{j=1}^N a_j\, z_{j}^{k+r} \overline{v}^{(l)}_{r} = \sum_{j=1}^N a_j\,  P_{\overline{\mathbf v}^{(l)}} (z_{j}) \, z_{j}^{k}
$$
for all $k \in {\N}_{0}$, and hence (\ref{eq:coneigenred}) is true.
The relation (\ref{eig}) is now a consequence of (\ref{eq:coneigenred}) observing that 
\begin{equation}\label{con} \sigma_{l} P_{{\mathbf v}^{(l)}}(z) = \sigma_l\sum_{k=0}^{\infty} v_k^{(l)} z^k = \sum_{k=0}^{\infty} \sum_{j=1}^{N} a_{j}  P_{\overline{\mathbf v}^{(l)}} (z_{j}) z_{j}^{k} z^{k} = \sum_{j=1}^{N} \frac{a_{j} P_{\overline{\mathbf v}^{(l)}} (z_{j})}{1- z_{j}z} 
\end{equation}
for $z\in\mathbb{D}$ by inserting $z=\bar{z}_{r}$, $r=1, \ldots , N$.
\end{proof}
\smallskip

Equation (\ref{con}) also shows that  that $P_{{\mathbf v}^{(l)}}(z)$ is a rational function whose numerator is a polynomial of degree at most $N-1$. Thus, in order to find the zeros of $P_{{\mathbf v}^{(l)}}(z)$ we only need to compute the $N-1$ zeros of the numerator in this rational representation.
Note that a similar idea of dimension reduction has been used by Beylkin and Monz\'on in \cite{BM05}. But in contrast to the above approach, they considered the rank reduction of a finite Hankel matrix. We combine our observations  with Theorem \ref{AAK} and obtain the following procedure to compute the desired approximation $\tilde{\mathbf f}$ in the $\ell^2$ norm. 

\smallskip

\noindent
{\bf Algorithm for sparse approximation of exponential sums.}

\smallskip
\noindent
{\bf Input:} samples $f_{k}$, $k=0, \ldots, M$ for sufficiently large $M\ge 2N-1$,\\
\null \hspace{12mm} target approximation error $\epsilon$.
\noindent
\begin{enumerate}
\item Find the parameters  $z_{j} \in {\mathbb D}$ and $a_{j} \in {\C}\setminus \{ 0 \}$, $j=1, \ldots, N$ of the exponential representation of ${\mathbf f}$ in (\ref{eq:expsum}) using a Prony-like method.
\item Solve the con-eigenproblem for the matrix ${\mathbf A}_{N}{\mathbf Z}_{N}$ and determine the largest singular value $\sigma_{K}$ with $\sigma_{K} < \epsilon$.
\item Compute the $K$ zeros ${z}_{j}^{(K)} \in {\mathbb D}$, $j=1, \ldots , K$,  of the con-eigenpolynomial $P_{{\mathbf v}^{(K)}}(z)$ of $\G_{\mathbf f}$ using its rational representation (\ref{con}).
\item Compute the coefficients $\tilde{a}_{j}$ by solving the minimization problem 
$$ \min_{\tilde{a}_{1}, \ldots , \tilde{a}_{K}} \| {\mathbf f}-\tilde{\mathbf f} \|_{\ell^{2}}^{2} = \min_{\tilde{a}_{1}, \ldots , \tilde{a}_{K}} \sum_{k=0}^{\infty} |f_{k} - \sum_{j=1}^{K} \tilde{a}_{j} ({z}_{j}^{(K)})^{k}|^{2}.
$$
\end{enumerate}
{\bf Output:} sequence $\tilde{\mathbf f}$ of the form (\ref{ff}) such that $\| {\mathbf f} - \tilde{\mathbf {f}} \|_{\ell^{2}} \le \sigma_{n}< \epsilon$.

\smallskip

\begin{remarks} \null $~$

\noindent
1. 
Since ${\mathbf A}_{N} {\mathbf Z}_{N}$ is con-diagonalizable by Theorem \ref{reduction}, it follows from Theorem 4.6.6 in \cite{HJ} that $\overline{{\mathbf A}_{N}{\mathbf Z}_{N}}{\mathbf A}_{N}{\mathbf Z}_{N}$ has only real nonnegative eigenvalues $\lambda_{j}$. Now, if $(\lambda_{j}, {\mathbf w}^{(j)})$ is an eigenpair of $\overline{{\mathbf A}_{N}{\mathbf Z}_{N}}{\mathbf A}_{N}{\mathbf Z}_{N}$, then ${\mathbf v}^{(j)}:= {\mathbf A}_{N} {\mathbf Z}_{N} {\mathbf w}^{(j)} + \sigma_{j} \overline{\mathbf w}^{(j)}$ is a con-eigenvector of ${\mathbf A}_{N} {\mathbf Z}_{N}$ to the con-eigenvalue $\sigma_{j} = \sqrt{\lambda_{j}}$, since
\begin{eqnarray*}
 {\mathbf A}_{N} {\mathbf Z}_{N} \overline{\mathbf v}^{(j)} &=& {\mathbf A}_{N} {\mathbf Z}_{N} (\overline{\mathbf A}_{N} \overline{\mathbf Z}_{N} \overline{\mathbf w}^{(j)} + \sigma_{j} {\mathbf w}^{(j)}) \\
 &=& \overline{ \overline{{\mathbf A}_{N}{\mathbf Z}_{N}}{\mathbf A}_{N}{\mathbf Z}_{N} {\mathbf w}^{(j)}} + \sigma_{j} {\mathbf A}_{N} {\mathbf Z}_{N} {\mathbf w}^{(j)} \\
&=& \sigma_{j}^{2} \overline{\mathbf w}^{(j)} + \sigma_{j} {\mathbf A}_{N} {\mathbf Z}_{N} {\mathbf w}^{(j)} = \sigma_{j} {\mathbf v}^{(j)}.
\end{eqnarray*}
Thus, to solve the con-eigenvalue problem in step 2 of the algorithm, we have to consider a usual eigenvalue-decomposition of $\overline{{\mathbf A}_{N}{\mathbf Z}_{N}}{\mathbf A}_{N}{\mathbf Z}_{N}$.

\noindent
2. Observing that the components $f_{k}$ have the form (\ref{eq:expsum}), the $\ell^{2}$-minimization problem in step 4 of the algorithm breaks down to a least squares problem 
with complex coefficients of the form
$$ \min_{\tilde{a}_{1}, \ldots , \tilde{a}_{K}} \left( 2{\mathrm Re}  \sum_{\nu=1}^{N} \sum_{j=1}^{K} \frac{\overline{a}_{\nu} \tilde{a}_{j}}{1 - \overline{z}_{\nu} z_{j}^{(K)} } + \sum_{j=1}^{K} \sum_{j'=1}^{K} \frac{\tilde{a}_{j} \overline{\tilde{a}}_{j'}}{ 1 - z_{j}^{(K)} \overline{z_{j}^{(K)}}} \right).
$$
If we are interested to find an optimal sequence $\tilde{\bf f}$ in the $\ell^{1}$-norm instead of the $\ell^{2}$-norm, then we have to replace the minimization problem in step 4 accordingly. This problem can be reformulated as a linear program but its solution is more expensive than solving a least squares problem.

\noindent
3. For a short summary of this section we also refer to \cite{PP16}.
\end{remarks}

\section{The AAK theory revisited}
\label{sec3} 
\setcounter{equation}{0}

In order to prove Theorem \ref{AAK}, we need some more inside information on infinite Hankel and Toeplitz matrices. We aim at 
showing the needed assertion using only tools from linear algebra and Fourier analysis. We hope that these insights will help to derive an analogous result on structured low-rank approximation of finite Hankel matrices in the future. The original result as well as all
further presentations of the AAK theory that we are aware of require fundamental theorems in complex analysis for approximation of meromorphic functions, as the Nehari theorem and the Beurling theorem, see e.g. \cite{AAK,Young,Peller, Chui}.

\subsection{Basic properties of infinite Hankel and Toeplitz matrices}

Let us start with summarizing some important properties of special infinite Hankel and Toeplitz matrices.
As shown in Section \ref{preliminaries}, we consider a Hankel matrix ${\mathbf \Gamma}_{\mathbf f}$ as in (\ref{gamma}) generated by ${\mathbf f} \in \ell^{1}$ such that ${\mathbf \Gamma}_{\mathbf f}: \ell^{p} \to \ell^{p}$ is bounded for $p=1$ and $p=2$.

\medskip

\noindent
{\bf Characterization of the kernel of an infinite Hankel matrix.}
Let  $ {\mathbf v}:=(v_k)_{k=0}^\infty$ be a sequence in $\ell^p$ and $p\in \{ 1, 2 \}$. We define the \textit{(forward) shift operator} $\sh:\ell^p \to \ell^p$ by
$$
	\sh{\mathbf v}:=(0,v_0,v_1,v_2,\ldots)
$$
and the \textit{backward shift operator} $\sh^*:\ell^p \to \ell^p $ by
$$
	\sh^*{\mathbf v}:=(v_1,v_2,v_3,\ldots).
$$
The \textit{shift invariant subspace} of $\ell^p$ generated by the sequence $ {\mathbf v} \in \ell^2$ (resp.\  $ {\mathbf v} \in \ell^{1} \subset \ell^2$) is denoted by
$$
	{\cal S}_{\mathbf v} := {\rm  clos}_{\ell^{2}}{\rm span} \;  \{\sh^k {\mathbf v} : k \in {\N}_{0} \}.
$$
Note that the $k$-th row of ${\mathbf \Gamma}_{\mathbf f}$ is the backward shift $(\sh^{*})^k {\mathbf f}$ of the first row. Due to the symmetry, the same holds for the columns. Thus  for ${\mathbf v} \in \ell^{2}$ we have
\begin{equation}\label{s} {\mathbf \Gamma}_{\mathbf f} \sh  {\mathbf v} = \left( \sum_{j=1}^{\infty} f_{k+j} \, v_{j-1} \right)_{k=0}^{\infty} = \left( \sum_{j=0}^{\infty
} f_{k+1+j} v_{j} \right)_{k=0}^{\infty} =  \sh^{*} {\mathbf \Gamma}_{\mathbf f} {\mathbf v}. 
\end{equation}
This commutator relation determines the structure of a Hankel operator and can be even used as a formal definition of ${\mathbf \Gamma}_{\mathbf f}$, see \cite{AAK}.

The following lemma gives a useful characterization of the kernel of ${\mathbf \Gamma}_{\mathbf f}$.

\begin{lemma}\label{shiftinv}
Let ${\mathbf f} := (f_{k})_{k=0}^{\infty}$ be a sequence in $\ell^{1}$ and ${\mathbf \Gamma}_{\mathbf f}$ the corresponding infinite Hankel matrix as above. Then the following assertions hold.
\begin{itemize}
\item[\textnormal{(i)}] The kernel space $\Ker \,  ({\mathbf \Gamma}_{\mathbf f}) := \{ {\mathbf v} \in \ell^{2}: \, {\mathbf \Gamma}_{\mathbf f} {\mathbf v} = {\mathbf 0} \}$  is $\sh$-invariant, i.e., for ${\mathbf v} \in \Ker ({\mathbf \Gamma}_{\mathbf f})$ we have ${\cal S}_{\mathbf v} \subset  \Ker ({\mathbf \Gamma}_{\mathbf f})$.
\item[\textnormal{(ii)}] A vector ${\mathbf v} \in \ell^{2}$ is in ${\mathrm ker} \,  ({\mathbf \Gamma}_{\mathbf f})$ if and only if 
${\mathbf f} \in ({\cal S}_{\overline{\mathbf v}})^{\perp}$.
\end{itemize}
\end{lemma}

\begin{proof}$~$
1. Let $ {\mathbf v} \in\Ker ({\mathbf \Gamma}_{\mathbf f})$. Then (\ref{s}) implies
$$
{\mathbf \Gamma}_{\mathbf f} \sh {\mathbf v} = \sh^*{\mathbf \Gamma}_{\mathbf f} {\mathbf v} = \sh^*{\mathbf 0} = {\mathbf 0},
$$
thus $\sh {\mathbf  v}$ is also in $\Ker ({\mathbf \Gamma}_{\mathbf f})$.

2. Using the definition of ${\cal S}_{\mathbf  v}$ we obtain
$$
\begin{aligned}
{\mathbf \Gamma}_{\mathbf f} {\mathbf v} = 0 &~\Leftrightarrow~ \sum_{k=0}^\infty f_{k+j}\, v_k = 0 & \forall~j \in {\N}_{0}\\
			 &~\Leftrightarrow~ \sum_{k=0}^\infty (S^j {\mathbf v})_k {f_{k}} = 0 & \forall~j \in {\N}_{0}\\
			 &~\Leftrightarrow~ \langle {\mathbf f} , \sh^j \bar{\mathbf v} \rangle_{\ell^{2}} = 0 & \forall~j \in {\N}_{0}\\
			 &~\Leftrightarrow~~{\mathbf f}\perp {\cal S}_{\overline{\mathbf v}}
\end{aligned}
$$
for every $ {\mathbf v}\in \ell^2$.
\end{proof}

Let us now come back to the sequence ${\mathbf f}$ of the special form (\ref{eq:expsum}) with $z_{j} \in {\D}$.
Then the structure of (con)-eigenvectors corresponding to the zero-con-eigenvalues of ${\mathbf \Gamma}_{\mathbf f}$ can be described as follows. 

\begin{theorem} \label{zerovectors}
Let ${\mathbf f}$ be a vector of the form {\rm (\ref{eq:expsum})}.
Then, ${\mathbf v} \in \ell^{2}$ satisfies ${\mathbf \Gamma}_{\mathbf f} \overline{\mathbf v} = {\mathbf 0}$ if and only if the corresponding Laurent polynomial satisfies $P_{\overline{\mathbf v}}(z_{j})=0$, for $j=1, \ldots , N$, where the $z_{j}$ are given in {\rm (\ref{eq:expsum})}.
\end{theorem}

\begin{proof} Observe first that $P_{\overline{\mathbf v}}(z)$ is well-defined for each $z \in \D$.
The assertion ${\mathbf \Gamma}_{\mathbf f} \overline{\mathbf v} = {\mathbf 0}$ implies
$$ 0 = ({\mathbf \Gamma}_{\mathbf f} \overline{\mathbf v})_{k} = \sum_{r=0}^{\infty} f_{k+r} \overline{v}_{r}
= \sum_{r=0}^{\infty} \sum_{j=1}^{N} a_{j} z_{j}^{k+r} \overline{v}_{r} 
= \sum_{j=1}^{N} a_{j} z_{j}^{k} \sum_{r=0}^{\infty} \overline{v}_{r} z_{j}^{r} = 
\sum_{j=1}^{N} a_{j} P_{\overline{\mathbf v}} (z_{j}) z_{j}^{k},
$$
for all $k \in {\N}_{0}$ and thus
$$ 0 = \sum_{r=0}^{\infty} \sum_{j=1}^{N} a_{j}  P_{\overline{\mathbf v}} (z_{j}) z_{j}^{r} z^{r} = \sum_{j=1}^{N} \frac{a_{j} P_{\overline{\mathbf v}} (z_{j})}{1- z_{j}z}
$$
for all $z \in {\D}$. 
Hence, $P_{\overline{\mathbf v}} (z_{j}) = 0$ for $j=1, \ldots , N$.
Conversely, $P_{\overline{\mathbf v}} (z_{j}) = 0$ obviously implies that this equation is satisfied.
\end{proof}

\begin{remark}
A result similar to the assertion of Theorem \ref{zerovectors} can be found e.g. in \cite{Young}, see Lemma 16.11 in the context of of the Adamyan-Arov-Krein-Theory \cite{AAK} for approximation of meromorphic functions in Hardy spaces.
\end{remark}

\medskip

\noindent
{\bf Trianguar Toeplitz matrices.}
For ${\mathbf g}=(g_{k})_{k=0}^{\infty} \in \ell^{p}$ $p \in \{ 1,2 \}$, we define the infinite triangular Toeplitz matrix ${\mathbf T}_{\mathbf g}$ by
$$ {\mathbf T}_{\mathbf g} :=
\left( \begin{array}{cccc} g_{0} & & & \\ g_{1} & g_{0}& & \\
g_{2}& g_{1} & g_{0} & \\ \vdots & \vdots & & \ddots \end{array} \right).$$
For ${\bf g} \in \ell^{1}$, ${\mathbf T}_{\mathbf g}$ determines a bounded operator ${\mathbf T}_{\mathbf g}: \ell^{\nu} \to \ell^{\nu}$ for $\nu \ge 1$ given by
$$ {\mathbf T}_{\mathbf g} {\mathbf v} := \left( \sum_{j=0}^{k} g_{k-j} v_{j} \right)_{k=0}^{\infty} = {\mathbf g} * {\mathbf v}, \qquad {\mathbf v} \in \ell^{\nu},$$ 
since 
$$ \| {\mathbf T}_{\mathbf g} {\mathbf v} \|_{{\nu}} = \| {\mathbf g} * {\mathbf v} \|_{{\nu}} \le \| {\mathbf g} \|_{{1}} \| {\mathbf v} \|_{{\nu}}$$
by Young's inequality. Similarly, for  ${\bf g} \in \ell^{2}$, ${\mathbf T}_{\mathbf g}: \ell^{1} \to \ell^{1}$ is bounded. We summarize some important properties of ${\mathbf T}_{\mathbf g}$ in the following two lemmas.

\begin{lemma}\label{TB}
For two sequences ${\mathbf f} \in \ell^{1}$ and  ${\mathbf g} \in \ell^{p}$ $p \in \{1,2\}$  we have
\begin{itemize}
\item[\textnormal{(1)}] The convolution ${\mathbf f}* {\mathbf g}$ is a sequence in $\ell^{p}$ and
$$
	{\mathbf T}_{{\mathbf f} * {\mathbf g}} = {\mathbf T}_{{\mathbf f}} \cdot {\mathbf T}_{{\mathbf g}} = 
	{\mathbf T}_{{\mathbf g}}\cdot {\mathbf T}_{{\mathbf f}}.
$$
The corresponding Fourier series  satisfy $P_{{\mathbf f}* {\mathbf g}}(e^{i\omega}) = 
P_{\mathbf f}(e^{i\omega}) \cdot P_{\mathbf g}(e^{i\omega})$ for all $\omega \in {\R}$.
\item[\textnormal{(2)}] For ${\bf g} \in \ell^{1}$, ${\mathbf \Gamma}_{\mathbf f} {\mathbf T}_{\mathbf g}$ is a bounded Hankel operator on $\ell^{\nu}$ for $\nu \ge 1$.
For ${\bf g} \in \ell^{2}$, ${\mathbf \Gamma}_{\mathbf f} {\mathbf T}_{\mathbf g}$ is a bounded Hankel operator on $\ell^{1}$.

\item[\textnormal{(3)}] We have ${\mathbf \Gamma}_{\mathbf f} {\mathbf T}_{\mathbf g} = {\mathbf T}_{\mathbf g}^{T} {\mathbf \Gamma}_{\mathbf f}$ .
\end{itemize}
\end{lemma}

\begin{proof} $~$
1. We observe  that for $l \ge k$
$$ ({\mathbf T}_{{\mathbf f}} {\mathbf T}_{{\mathbf g}})_{l,k} = \sum_{r=0}^{l-k} {f}_{(l-k)-r} g_{r}
= \sum_{r=0}^{l-k} f_{r} g_{(l-k)-r} = ({\mathbf T}_{\mathbf f} {\mathbf T}_{{\mathbf g}})_{l,k} = ({\mathbf f} * {\mathbf g})_{l-k} $$
while $({\mathbf T}_{{\mathbf f}} {\mathbf T}_{{\mathbf g}})_{l,k} =0$ for $l < k$. Young's inequality ensures that $({\mathbf f} * {\mathbf g}) \in \ell^{p}$ and thus the product of Toeplitz operators is a bounded operator on $\ell^{\nu}$, $\nu \ge 1$ for $p=1$ and on $\ell^{1}$ for $p=2$. The relation for the corresponding Fourier series follows by the convolution theorem.

2. Since the $j$-th row of ${\mathbf \Gamma}_{\mathbf f}$ is $({\mathbf S}^*)^j {\mathbf f}$ and the $k$-th column of ${\mathbf T}_{\mathbf g}$ is ${\mathbf S}^k {\mathbf g}$, it follows that 
$$
({\mathbf \Gamma}_{\mathbf f} {\mathbf T}_{\mathbf g})_{j,k} = (({\mathbf S}^*)^j {\mathbf f})^T ({\mathbf S}^k {\mathbf g}) = {\mathbf f}^T ({\mathbf S}^{j+k}{\mathbf g}),
$$
thus the entries of ${\mathbf \Gamma}_{\mathbf f} {\mathbf T}_{\mathbf g}$ only depend on the sum of its indices.  Therefore, ${\mathbf \Gamma}_{\mathbf f} {\mathbf T}_{\mathbf g}$ has again Hankel structure. For $p=1$, the obtained Hankel matrix is generated by ${\mathbf \Gamma}_{\mathbf f} {\mathbf g} \in \ell^{1}$, and
for $p=2$ we get ${\mathbf \Gamma}_{\mathbf f} {\mathbf g} \in \ell^{2}$ by Young's inequality.

3. Similarly, since  the $j$-th row of ${\mathbf T}_{\mathbf g}$ is ${\mathbf S}^j {\mathbf g}$ and the $k$-th column of ${\mathbf \Gamma}_{\mathbf f}$ is $({\mathbf S}^*)^k {\mathbf f}$ we obtain 
$$
({\mathbf T}_{\mathbf g}^T {\mathbf \Gamma}_{\mathbf f})_{j,k} = ({\mathbf S}^j{\mathbf g})^T (({\mathbf S}^*)^k {\mathbf f}) = ({\mathbf S}^{j+k} {\mathbf g})^T {\mathbf f} = {\mathbf f}^T ({\mathbf S}^{j+k} {\mathbf g}) = ({\mathbf \Gamma}_{\mathbf f} {\mathbf T}_{\mathbf g})_{j,k}.
$$ 
\end{proof}

\begin{lemma}\label{toeplitz}$~$
For some $K \in {\N}_{0}$ let ${\mathbf b} = (b_{k})_{k=0}^{\infty}$ be given by the Blaschke product
\begin{equation}\label{bla}
 B(e^{i\omega}) = \sum_{k=0}^{\infty} b_{k} e^{i\omega k} := \left\{ \begin{array}{ll} \prod_{j=1}^{K} \frac{e^{i\omega}-\alpha_{j}}{1-\overline{\alpha}_{j} e^{i\omega}} & K>0,\\ 
1 & K=0,\end{array} \right. 
\end{equation}
where $\alpha_{1}, \ldots , \alpha_{K} \in {\D}$.
Then ${\bf b} \in \ell^{1}$ and the infinite triangular Toeplitz matrix ${\mathbf T}_{\mathbf b}$ generated by ${\mathbf b}$ satisfies the following properties.
\begin{itemize}
\item[\textnormal{(1)}] ${\mathbf T}_{\mathbf b}^*{\mathbf T}_{\mathbf b} = {\mathbf I}$, i.e. ${\mathbf T}_{\mathbf b}^{*}$ is the left inverse of ${\mathbf T}_{\mathbf b}$.
\item[\textnormal{(2)}] The operator ${\mathbf T}_{\mathbf b}:\ell^{p} \to \ell^{p}$ has the norm $\|{\mathbf T}_{\mathbf b}\|_{\ell^{p} \to \ell^{p}} = 1$ for $p \in \{ 1,2\}$.
\item[\textnormal{(3)}] Let ${\mathbf \Gamma}_{\mathbf f}$ be an infinite Hankel matrix being generated by ${\mathbf f} \in \ell^{1}$.
Let $\sigma_n ({\mathbf \Gamma}_{\mathbf f})$ and $\sigma_n ({\mathbf \Gamma}_{\mathbf f} {\mathbf T}_{\mathbf b})$ be the $n$-th singular values of  ${\mathbf \Gamma}_{\mathbf f}$ and ${\mathbf \Gamma}_{\mathbf f} {\mathbf T}_{\mathbf b}$ being ordered decreasingly as $\sigma_{0}({\mathbf \Gamma}_{\mathbf f}) \ge \sigma_{1}({\mathbf \Gamma}_{\mathbf f}) \ge \ldots $ and $\sigma_{0}({\mathbf \Gamma}_{\mathbf f} {\mathbf T}_{\mathbf b}) \ge \sigma_{1}({\mathbf \Gamma}_{\mathbf f} {\mathbf T}_{\mathbf b}) \ge \ldots $.
Then, for all $n \in {\N}_{0}$, we have
$$ \sigma_n({\mathbf \Gamma}_{\mathbf f} {\mathbf T}_{\mathbf b}) \leq \sigma_n({\mathbf \Gamma}_{\mathbf f}).
$$
	\end{itemize}
\end{lemma}

\begin{proof}$~$
1. Obviously, ${\mathbf T}_{\mathbf b}^{*} {\mathbf T}_{\mathbf b}$ is hermitian. For the $(l,k)$-th entry of ${\mathbf T}_{\mathbf b}^{*} {\mathbf T}_{\mathbf b}$
we obtain for $l \ge k$
$$ ({\mathbf T}_{\mathbf b}^{*} {\mathbf T}_{\mathbf b})_{l,k} = \sum_{j=l}^{\infty} \overline{b}_{j-l} b_{j-k}
= \sum_{j=0}^{\infty} \overline{b}_{j} b_{j+(l-k)}. $$
The coefficients $b_{k}$ are the Fourier coefficients of $B(e^{i \omega})$,
$$ b_{k} = \frac{1}{2\pi} \int_{0}^{2\pi} B(e^{i\omega}) \, e^{-i\omega k} d \omega, \qquad k=0, 1,2, \ldots . $$
Thus,
\begin{eqnarray*}
\sum_{j=0}^{\infty} \overline{b}_{j} b_{j+(l-k)} &=& \frac{1}{2\pi} \sum_{j=0}^{\infty} \overline{b}_{j} \int_{0}^{2\pi}  B(e^{i\omega}) e^{-i\omega(j+l-k)} d\omega \\
&=& 
\frac{1}{2\pi} \int_{0}^{2\pi}  B(e^{i\omega}) e^{-i\omega(l-k)}  \sum_{j=0}^{\infty} \overline{b}_{j} e^{-i\omega j} d\omega \\
&=& \frac{1}{2\pi} \int_{0}^{2\pi}  |B(e^{i\omega})|^{2} e^{-i\omega(l-k)} d \omega = \delta_{l,k}
\end{eqnarray*}
since
$$ |B(e^{i\omega})|^{2} = \prod_{j=1}^{K} \left( \frac{e^{i\omega}-\alpha_{j}}{1-\overline{\alpha_{j}} e^{i\omega}} \right) \left( \frac{e^{-i\omega}-\overline{\alpha_{j}}}{1-{\alpha_{j}} e^{-i\omega}} \right) = 1. $$
2. Now the second assertion follows directly from the first.

3. Using the definition of the singular value and the properties of ${\mathbf T}_{\mathbf b}$ we obtain for $p \in \{ 1,2 \}$
\begin{eqnarray*}
\sigma_n( {\mathbf \Gamma}_{\mathbf f}) &=& \min \{ \| {\mathbf \Gamma}_{\mathbf f}- {\mathbf R}\| ~:~ {\mathbf R} \in {\cal L}(\ell^{p}), \, \rank({\mathbf R}) \leq n \}\\
							 &=& \min \{ \|{\mathbf \Gamma}_{\mathbf f}- {\mathbf R}\| \|{\mathbf T}_{\mathbf b}\| ~:~ {\mathbf R} \in {\cal L}(\ell^{p}), \, \rank({\mathbf R}) \leq n \}\\
							 & \geq& \min \{ \|({\mathbf \Gamma}_{\mathbf f}- {\mathbf R}) {\mathbf T}_{\mathbf b}\| ~:~ {\mathbf R} \in {\cal L}(\ell^{p}), \, \rank({\mathbf R}) \leq n \}\\
							 &=& \min \{ \| {\mathbf \Gamma}_{\mathbf f} {\mathbf T}_{\mathbf b}- {\mathbf R}{\mathbf T}_{\mathbf b}\| ~:~ {\mathbf R} \in {\cal L}(\ell^{p}), \, \rank({\mathbf R}) \leq n \}\\
							 &=& \min \{ \| {\mathbf \Gamma}_{\mathbf f} {\mathbf T}_{\mathbf b}-\tilde{\mathbf R}\| ~:~ \tilde{\mathbf R} \in {\cal L}(\ell^{p}), \, \rank(\tilde{\mathbf R}) \leq n \}\\
							 &=& \sigma_{n} ({\mathbf \Gamma}_{\mathbf f} {\mathbf T}_{\mathbf b}),
\end{eqnarray*}
		since $\rank({\mathbf R}{\mathbf T}_{\mathbf b})$ is still at most $n$. Here ${\cal L}(\ell^{p} )$ denotes the set of all linear operators from $\ell^{p}$ to $\ell^{p}$.
\end{proof}

\medskip

\noindent
{\bf Construction of infinite Hankel matrices with special properties}.
Next, we will construct an infinite Hankel matrix with operator norm $1$ that possesses a predetermined con-eigenvector ${\mathbf v} \in \ell^{1}$ to the con-eigenvalue $1$. For that purpose, we first need to understand the image of an infinite Hankel matrix.

\begin{lemma}\label{lem1}
For given sequences ${\mathbf f} \in \ell^{1}$ and ${\mathbf v} \in \ell^{p}$, $p \in \{ 1, 2 \}$ with corresponding Fourier series 
$P_{\mathbf f}(e^{i\omega})$ and $P_{\mathbf v}(e^{i\omega})$ the vector ${\mathbf w} = (w_{k})_{k=0}^{\infty}$ obtained by
$$ {\mathbf w} = {\mathbf \Gamma}_{\mathbf f} {\mathbf v}$$
satisfies
$$ w_{k} = \frac{1}{2\pi} \int_{0}^{2\pi} P_{\mathbf f} ({\mathrm e}^{{\mathrm i} t}) P_{\mathbf v} ({\mathrm e}^{-{\mathrm i} t})
{\mathrm e}^{- {\mathrm i} t k } {\mathrm d} t, \qquad k \in {\N}_{0}. $$
\end{lemma}

\begin{proof}
Let $P_{\mathbf w} (e^{i\omega}) := \sum_{k=0}^{\infty} w_{k} e^{i \omega k}$. Then, on the one hand, we find
$$ P_{\mathbf w}(e^{i\omega}) = \sum_{k=0}^{\infty} \sum_{j=0}^{\infty} f_{k+j} v_{j} e^{i \omega k} = \sum_{j=0}^{\infty} \sum_{k=j}^{\infty} f_{k} v_{j} e^{i \omega (k-j)}. $$
On the other hand,
$$ P_{\mathbf f}(e^{i\omega}) P_{\mathbf v}(e^{-i\omega}) = \sum_{k=0}^{\infty} \sum_{j=0}^{\infty} f_{k} v_{j} e^{i \omega (k-j)} = P_{\mathbf w}(e^{i\omega}) +
\sum_{j=0}^{\infty} \sum_{k=0}^{j-1} f_{k} v_{j} e^{i \omega(k-j)}, $$
where in the second sum occur only negative powers of $e^{i\omega}$.
Hence, $P_{\mathbf w}({\mathrm e}^{{\mathrm i} t})$ possesses the Fourier coefficients
$$ w_{k}= \frac{1}{2\pi} \int_{0}^{2\pi} P_{\mathbf f}({\mathrm e}^{{\mathrm i} t}) P_{\mathbf v}({\mathrm e}^{-{\mathrm i} t})
{\mathrm e}^{-{\mathrm i} t k} {\mathrm d} t $$
for $k \in \N_{0}$.
\end{proof}

\noindent
Now we consider the construction of a special infinite Hankel matrix with operator norm $1$.

\begin{lemma}\label{lem2}
Let ${\mathbf v} \in \ell^{1}$ be given with the corresponding Fourier series  $P_{\mathbf v} (e^{i\omega}) $. Assume that $P_{\mathbf v} (e^{i\omega}) \neq 0$ for all $\omega \in [0, 2 \pi)$. Further, let ${\mathbf w}= (w_{k})_{k=0}^{\infty}$ be given by
$$ w_{k} := \frac{1}{2\pi} \int_{0}^{2\pi} \frac{P_{\mathbf v} ({\mathrm e}^{{\mathrm i} t})}{P_{\overline{\mathbf v}}({\mathrm e}^{-{\mathrm i} t})} {\mathrm e}^{-{\mathrm i} t k} {\mathrm d} t, \qquad k \in \N_{0}. $$
 Then ${\mathbf w} \in \ell^{2}$, and $\| {\mathbf w} \|_{{2}} =1$. Further, the Hankel operator ${\mathbf \Gamma}_{\mathbf w}$ satisfies ${\mathbf \Gamma}_{\mathbf w} \overline{\mathbf v} = {\mathbf v}$ and
$$ \| {\mathbf \Gamma}_{\mathbf w} \|_{\ell^{2} \to \ell^{2}} := \sup_{{\mathbf u} \in \ell^{2}\setminus \{ {\mathbf 0 \}}} \frac{\| {\mathbf \Gamma}_{\mathbf w} {\mathbf u} \|_{2}}{ \| {\mathbf u} \|_{2} } = \frac{ \| {\mathbf \Gamma}_{\mathbf w} \overline{\mathbf v} \|_{2}}{ \| {\mathbf v} \|_{2} } = 1. $$
\end{lemma}

\begin{proof}
First we observe that by Parseval's identity
$$ \| {\mathbf  w} \|_{2}^{2} = \sum_{k=0}^{\infty} |w_{k}|^{2} \le \left\| \frac{P_{\mathbf v} ({\mathrm e}^{{\mathrm i} \cdot})}{P_{\overline{\mathbf v}}({\mathrm e}^{-{\mathrm i} \cdot})} \right\|_{L^{2}([0,2\pi))}^{2} = \frac{1}{2\pi} \int_{0}^{2\pi}
 \left| \frac{P_{\mathbf v} ({\mathrm e}^{{\mathrm i} t})}{P_{\overline{\mathbf v}}({\mathrm e}^{-{\mathrm i} t})} \right|^{2} {\mathrm d} t =1$$
 and thus ${\mathbf w} \in \ell^{2}$. 
Further, we obtain
 \begin{eqnarray*}
 ({\mathbf \Gamma}_{\mathbf w} \overline{\mathbf v} )_{k} &=& \sum_{j=0}^{\infty} w_{k+j} \bar{v}_{j} = \sum_{j=0}^{\infty} \frac{1}{2\pi} \int_{0}^{2\pi} \frac{P_{\mathbf v} ({\mathrm e}^{{\mathrm i} t})}{P_{\overline{\mathbf v}}({\mathrm e}^{-{\mathrm i} t})} {\mathrm e}^{-{\mathrm i} t (k+j)} 
 \bar{v}_{j} {\mathrm d} t \\
 &=& \frac{1}{2\pi} \int_{0}^{2\pi} \frac{P_{\mathbf v} ({\mathrm e}^{{\mathrm i} t})}{P_{\overline{\mathbf v}}({\mathrm e}^{-{\mathrm i} t})} {\mathrm e}^{-{\mathrm i} t k}\sum_{j=0}^{\infty} \bar{v}_{j} {\mathrm e}^{-{\mathrm i} t j} {\mathrm d} t
 = \frac{1}{2\pi} \int_{0}^{2\pi} P_{\mathbf v} ({\mathrm e}^{{\mathrm i} t}) {\mathrm e}^{-{\mathrm i} t k} {\mathrm d} t = v_{k}.
 \end{eqnarray*}
 for all $k \in \N_{0}$ and thus ${\mathbf \Gamma}_{\mathbf w} \overline{\mathbf v} = {\mathbf v}$. The norm of ${\mathbf \Gamma}_{\mathbf w}$ is indeed equal to $1$ since for arbitrary ${\mathbf u} \in \ell^{2}$ it follows by Lemma \ref{lem1} and Parseval's identity
\begin{eqnarray*}
\| {\mathbf \Gamma}_{\mathbf w} \overline{\mathbf u} \|_{2}^{2} 
&=& \sum_{k=0}^{\infty} \left| \frac{1}{2\pi} \int_{0}^{2\pi} P_{\mathbf w}({\mathrm e}^{{\mathrm i} t}) P_{\overline{\mathbf u}} ({\mathrm e}^{-{\mathrm i} t}) \, {\mathrm e}^{-{\mathrm i}k t} {\mathrm d}t \right|^{2} \\
&\le& \sum_{k=-\infty}^{\infty} \left| \frac{1}{2\pi} \int_{0}^{2\pi}  \frac{P_{\mathbf v} ({\mathrm e}^{{\mathrm i} t})}{P_{\overline{\mathbf v}} ({\mathrm e}^{-{\mathrm i} t})}  P_{\overline{\mathbf u}} ({\mathrm e}^{-{\mathrm i} t}) \, {\mathrm e}^{-{\mathrm i}k t} {\mathrm d}t \right|^{2} \\
&=& \frac{1}{2\pi} \int_{0}^{2\pi}  \left| \frac{P_{\mathbf v} ({\mathrm e}^{{\mathrm i} t})}{P_{\overline{\mathbf v}} ({\mathrm e}^{-{\mathrm i} t})} \right|^{2} | P_{\overline{\mathbf u}} ({\mathrm e}^{-{\mathrm i} t})|^{2} \, {\mathrm d} t \\
&\le& \sup_{t \in [0, 2\pi)} \left| \frac{P_{\mathbf v} ({\mathrm e}^{{\mathrm i} t})}{P_{\overline{\mathbf v}} ({\mathrm e}^{-{\mathrm i} t})} \right|^{2} \, \frac{1}{2\pi}
\int_{0}^{2\pi}  | P_{\overline{\mathbf u}} ({\mathrm e}^{-{\mathrm i} t})|^{2} \, {\mathrm d} t =  \sum_{k=-\infty}^{\infty} | \bar{u}_{k}|^{2} = \| {\mathbf u} \|_{2}^{2},
\end{eqnarray*}
and thus the assertion holds.
\end{proof}

This result immediately implies also $\| {\bf \Gamma}_{\bf w} {\bf u}\|_{2} \le \| {\bf w} \|_{2} \| {\bf u} \|_{1} = \| {\bf u} \|_{1}$ for all ${\bf u} \in \ell^{1}$ by Young's inequality.
\medskip

\subsection{Proof of the AAK-Theorem for Hankel matrices with finite rank}

Let us come back to the Hankel matrix ${\mathbf \Gamma}_{\mathbf f}$ of rank $N$ generated by the the sequence ${\mathbf f}$ of the form (\ref{eq:expsum})
with $1 > |z_{1}| \ge \cdots \ge |z_{N}| >0$, and with con-eigenvectors ${\mathbf v}^{(l)}$, $l=0, \ldots , N-1$ corresponding to the nonzero con-eigenvalues (resp.\ singular values) $\sigma_{0} \ge  \sigma_{1} \ldots \ge \sigma_{N-1} >0$.
As shown in (\ref{con}), the  Laurent polynomial corresponding to ${\mathbf v}^{(l)}$ has the form
$$ P_{{\mathbf v}^{(l)}}(z) = \sum_{j=0}^{\infty} v_{j}^{(l)} z^{j} 
= \sum_{j=1}^{N} \frac{a_{j} P_{\overline{\mathbf v}^{(l)}}(z_{j})}{1-z_{j}z} =\frac{{q}^{(l)}(z)}{z^{N} P(z^{-1})}$$
with ${q}^{(l)}(z)$ being a polynomial of degree $N-1$ and with the Prony polynomial $P(z)$ in (\ref{prony}).

We want to show now that for each single nonzero singular value $\sigma_{K}$ of $\G_{\mathbf f}$ the Laurent series of the corresponding con-eigenvector ${\mathbf v}^{(K)}$ possesses exactly $K$ zeros in $\D$, and moreover, that these zeros $z_{1}^{(K)}, \ldots , z_{K}^{(K)}$ can be used to construct a new Hankel matrix $\G_{\tilde{\mathbf f}}$ of rank $K$ with $\tilde{\mathbf f}$ of the form (\ref{eq:AAK1}) and $\| \G_{{\mathbf f}} -\G_{\tilde{\mathbf f}} \| = \sigma_{K}$.
\medskip

The above relation implies that the zeros of $P_{{\mathbf v}^{(K)}}(z)$ are the $N-1$ zeros of ${q}^{(K)}(z)$. 
Let $n_{K}$ denote the number of zeros of ${q}^{(K)}$ in $\D$, where $0 \le n_{K} \le N-1$.
We show first that $n_{K} \le K$.

We can write
$$ {q}^{(K)}(z) = \prod_{j=1}^{n_{K}} (z-\alpha_{j}) \, \prod_{j=n_{K}+1}^{N-1} (z-\beta_{j})$$
with $|\alpha_{j}| < 1$ and $|\beta_{j}| \ge 1$. Now, let
\begin{equation}\label{u} P_{{\mathbf u}^{(K)}}(z) := \frac{1}{\sigma_{K}} \frac{ \prod\limits_{j=1}^{n_{K}} (1-\overline{\alpha}_{j}z) \, \prod\limits_{j=n_{K}+1}^{N-1} (z-\beta_{j})}{z^{N} p(z^{-1}) }.
\end{equation}
Then  $P_{{\mathbf u}^{(K)}}(z)$ has no zeros in $\D$ and defines a sequence ${\mathbf u}^{(K)} = (u^{(l)}_{r})_{r=0}^{\infty} \in \ell^{1}$ via
$$ P_{{\mathbf u}^{(K)}}(z) = \sum_{r=0}^{\infty} u_{r}^{(K)} z^{r}, \quad z\in\D. $$
Denoting by 
\begin{equation}
\label{blal}
B^{(K)}(e^{i\omega}) := \prod_{j=1}^{n_{K}} \frac{e^{i\omega} - \alpha_{j}}{1- \overline{\alpha}_{j} e^{i\omega}} = \sum_{k=0}^{\infty} b_{k}^{(K)} e^{i\omega k}
\end{equation}
the Blaschke product on the unit circle, it follows that 
$$ P_{{\mathbf v}^{(K)}} (e^{i\omega}) = B^{(K)}(e^{i\omega}) \, P_{{\mathbf u}^{(K)}} (e^{i\omega}),$$
or  equivalently,
\begin{equation}\label{toep1} {\mathbf v}^{(K)} = {\mathbf T}_{{\mathbf b}^{(K)}} {\mathbf u}^{(K)}, \end{equation}
where ${\mathbf T}_{{\mathbf b}^{(K)}}$ denotes the triangular Toeplitz matrix  corresponding to $B^{(K)}(e^{i\omega})$. Now we can show

\begin{theorem}\label{aak2}$~$
Let ${\mathbf \Gamma}_{\mathbf f}$  be the infinite Hankel matrix of finite rank $N$ generated by ${\mathbf f} = (f_{k})_{k=0}^{\infty}$ of the form {\rm (\ref{eq:expsum})} and with nonzero singular values $\sigma_{0} \ge \sigma_{1} \ge \ldots 
\ge \sigma_{N-1} > 0$.  Further, let $(\sigma_{K}, {\mathbf v}^{(K)})$ be the $K$-th con-eigenpair of ${\mathbf \Gamma}_{\mathbf f}$  with $\sigma_{K} \neq \sigma_{k}$ for $K \neq k$. Let ${\mathbf T}_{{\mathbf b}^{(K)}}$  be the Toeplitz matrix corresponding to ${\mathbf v}^{(K)}$ being defined by $B^{(K)}(e^{i\omega})$ as in {\rm (\ref{blal})}.
Then 
${\mathbf \Gamma}_{\mathbf f} {\mathbf T}_{{\mathbf b}^{(K)}}$ possesses the singular value $\sigma_{K}$ with multiplicity at least $n_{K}+1$, where $n_{K}$ denotes the number of zeros of $P_{{\mathbf v}^{(K)}}$ in $\D$. In particular, we have $n_{K} \le K$. 
\end{theorem}

\begin{proof}$~$
1. Considering the Blaschke product in (\ref{blal}), we define its partial products by
$$ B^{(K)}_{j}(e^{i\omega}) := \sum_{r=0}^{\infty} (b_{j}^{(K)})_{r} e^{i\omega r} = \prod_{k=1}^{j} \frac{e^{i\omega}-\alpha_{k}}{1- \overline{\alpha}_{k} e^{i\omega}}, \qquad j=1, \ldots , n_{K}, $$
 where
$\alpha_{1}, \ldots , \alpha_{n_{K}}$ are the zeros of $P_{{\mathbf v}^{(K)}}(z)$ inside $\D$.
We employ the notation ${\mathbf T}_{{\mathbf b}^{(K)}/{\mathbf b}_{j}^{(K)}}$ for the triangular Toeplitz matrix
generated by the sequence of Fourier coefficients of $\prod_{k=j+1}^{n_{K}} \frac{e^{i\omega}-\alpha_{k}}{1- \overline{\alpha}_{k} e^{i\omega}}$
such that 
$$ {\mathbf T}_{{\mathbf b}^{(K)}} = {\mathbf T}_{{\mathbf b}^{(K)}/{\mathbf b}_{j}^{(K)}} \cdot {\mathbf T}_{{\mathbf b}_{j}^{(K)}}. $$
We show now that the $n_{K}+ 1$ vectors ${\mathbf v}^{(K)}, \, {\mathbf T}_{{\mathbf b}_{1}^{(K)}}^{*} {\mathbf v}^{(K)}, \ldots , {\mathbf T}_{{\mathbf b}_{n_{K}}^{(K)}}^{*} {\mathbf v}^{(K)}$ are linearly independent singular vectors of 
${\mathbf \Gamma}_{\mathbf f} {\mathbf T}_{{\mathbf b}^{(K)}}$ to the singular value $\sigma_{K}$.
For $j =0, \ldots , n_{K}$ (with ${\mathbf T}_{{\mathbf b}_{0}^{(K)}} := {\mathbf I}$) we obtain by Lemma \ref{TB} and Lemma \ref{toeplitz}
\begin{eqnarray*}
& &  ({\mathbf \Gamma}_{\mathbf f} {\mathbf T}_{{\mathbf b}^{(K)}})^{*} {\mathbf \Gamma}_{\mathbf f} {\mathbf T}_{{\mathbf b}^{(K)}} {\mathbf T}_{{\mathbf b}_{j}^{(K)}}^{*} {\mathbf v}^{(K)} \\
&=& {\mathbf T}_{{\mathbf b}^{(K)}}^{*} {\mathbf \Gamma}_{\mathbf f}^{*} {\mathbf \Gamma}_{\mathbf f} {\mathbf T}_{{\mathbf b}^{(K)}/{\mathbf b}^{(K)}_{j}} {\mathbf T}_{{\mathbf b}_{j}^{(K)}} {\mathbf T}_{{\mathbf b}_{j}^{(K)}}^{*} {\mathbf T}_{{\mathbf b}_{j}^{(K)}} {\mathbf T}_{{\mathbf b}^{(K)}/{\mathbf b}^{(K)}_{j}} {\mathbf u}^{(K)} \\
&=& {\mathbf T}_{{\mathbf b}^{(K)}}^{*} {\mathbf \Gamma}_{\mathbf f}^{*} {\mathbf T}_{{\mathbf b}^{(K)}/{\mathbf b}^{(K)}_{j}}^{T} {\mathbf \Gamma}_{\mathbf f} {\mathbf T}_{{\mathbf b}_{j}^{(K)}} {\mathbf T}_{{\mathbf b}^{(K)}/{\mathbf b}^{(K)}_{j}} {\mathbf u}^{(K)} \\
&=& {\mathbf T}_{{\mathbf b}^{(K)}}^{*} {\mathbf \Gamma}_{\mathbf f}^{*} {\mathbf T}_{{\mathbf b}^{(K)}/{\mathbf b}^{(K)}_{j}}^{T} {\mathbf \Gamma}_{\mathbf f} {\mathbf v}^{(K)} \\
&=& \sigma_{K} {\mathbf T}_{{\mathbf b}^{(K)}}^{*} {\mathbf \Gamma}_{\mathbf f}^{*} {\mathbf T}_{{\mathbf b}^{(K)}/{\mathbf b}^{(K)}_{j}}^{T}  \overline{\mathbf v}^{(K)} \\
&=& \sigma_{K} {\mathbf T}_{{\mathbf b}^{(K)}}^{*} {\mathbf \Gamma}_{\mathbf f}^{*} \overline{{\mathbf T}_{{\mathbf b}^{(K)}/{\mathbf b}^{(K)}_{j}}^{*}}  \overline{{\mathbf T}_{{\mathbf b}^{(K)}/{\mathbf b}^{(K)}_{j}}} \overline{{\mathbf T}_{{\mathbf b}^{(K)}_{j}}} \overline{\mathbf u}^{(K)} \\
&=& \sigma_{K} {\mathbf T}_{{\mathbf b}^{(K)}}^{*} {\mathbf \Gamma}_{\mathbf f}^{*} \overline{{\mathbf T}_{{\mathbf b}^{(K)}_{j}}} \overline{\mathbf u}^{(K)} \\
&=& \sigma_{K}  {\mathbf \Gamma}_{\mathbf f}^{*} \overline{{\mathbf T}_{{\mathbf b}_{j}^{(K)}}} \overline{{\mathbf T}_{{\mathbf b}^{(K)}/{\mathbf b}_{j}^{(K)}}} \overline{{\mathbf T}_{{\mathbf b}^{(K)}_{j}}} \overline{\mathbf u}^{(K)} \\
&=& \sigma_{K} {\mathbf T}_{{\mathbf b}_{j}^{(K)}}^{*} {\mathbf \Gamma}_{\mathbf f}^{*} \overline{\mathbf v}^{(K)} 
= \sigma_{K}^{2} {\mathbf T}_{{\mathbf b}_{j}^{(K)}}^{*} {\mathbf v}^{(K)}.
\end{eqnarray*}
Moreover, the vectors ${\mathbf v}^{(K)}, \, {\mathbf T}_{{\mathbf b}_{1}^{(K)}}^{*} {\mathbf v}^{(K)}, \ldots , {\mathbf T}_{{\mathbf b}_{n_{K}}^{(K)}}^{*} {\mathbf v}^{(K)}$ are linearly independent, since 
$$ \sum_{j=0}^{n_{K}} \gamma_{j} {\mathbf T}_{{\mathbf b}_{j}^{(K)}}^{*} {\mathbf v}^{(K)} = {\mathbf 0}$$
is by Lemma \ref{TB}(3) equivalent with 
$$ \G_{{\mathbf v}^{(K)}} \left( \sum_{j=0}^{n_{K}} \gamma_{j} {\mathbf b}_{j}^{(K)}\right) = {\mathbf 0},$$
i.e., $\left( \sum_{j=0}^{n_{K}} \gamma_{j} {\mathbf b}_{j}^{(K)}\right)$ is a zero-(con)-eigenvector of $\G_{{\mathbf v}^{(K)}}$.
Thus, by Theorem \ref{zerovectors} and (\ref{eq:coneigenred}), the Laurent polynomial 
$$ \sum_{r=0}^{\infty} \sum_{j=0}^{n_{K}} \gamma_{j} ({\mathbf b}_{j}^{(K)})_{r} z^{r} = \gamma_{0} + \sum_{j=1}^{n_{K}} \gamma_{j} \prod_{k=1}^{j} \frac{z-\alpha_{k}}{1- \overline{\alpha}_{k} z} $$
possesses all zeros $z_{1}, \ldots, z_{N}$. Since $n_{K} \le N-1$, we conclude that $\gamma_{0} = \ldots = \gamma_{n_{K}} =0$.
Therefore, ${\mathbf \Gamma}_{\mathbf f} {\mathbf T}_{{\mathbf b}^{(K)}}$ possesses the singular value $\sigma_{K}$ with multiplicity at least $n_{K}+1$.
On the  other hand, since $\sigma_{K}({\mathbf \Gamma}_{\mathbf f} {\mathbf T}_{{\mathbf b}^{(K)}}) \le \sigma_{K}({\mathbf \Gamma}_{\mathbf f})$ by Lemma \ref{toeplitz}, it follows that $n_{K} \le K$.
\end{proof}
\medskip

In the next step, we construct a sequence $\tilde{\mathbf f}^{(K)} = {\mathbf f} - {\mathbf g}^{(K)}$
such that $\rank({\mathbf \Gamma}_{{\mathbf f} - {\mathbf g}^{(K)}}) = K$ and $\| {\mathbf \Gamma}_{{\mathbf g}^{(K)}} \| = \| {\mathbf \Gamma}_{{\bf f} - \tilde{\bf f}^{(K)}} \| = \sigma_{K}$. 

Let $ {\mathbf \Gamma}_{{\mathbf g}^{(K)}}$ be the Hankel matrix generated by ${\mathbf g}^{(K)} = (g_{l}^{(K)})_{l=0}^{\infty}$ with 
\begin{equation} \label{g} 
g_{l}^{(K)} := \frac{\sigma_{K}}{2\pi} \int_{0}^{2\pi} \frac{P_{{\mathbf v}^{(K)}} ({\mathrm e}^{{\mathrm i} t})}{P_{\overline{\mathbf v}^{(K)}}({\mathrm e}^{-{\mathrm i} t})} {\mathrm e}^{-{\mathrm i} t l} {\mathrm d} t 
\end{equation}
for $l \in \N_{0}$.
Then by Lemma \ref{lem2} it follows that $\| {\mathbf \Gamma}_{{\mathbf g}^{(K)}} \|_{\ell^{2} \to \ell^{2}} = \sigma_{K}$ and ${\mathbf \Gamma}_{{\mathbf g}^{(K)}} \overline{\mathbf v}^{(K)} = \sigma_{K} {\mathbf v}^{(K)}$.

Now we can show

\begin{theorem}\label{theo}
Let ${\mathbf \Gamma}_{\mathbf f}$ be a Hankel operator of finite rank $N$ generated by ${\mathbf f} = (f_{k})_{k=0}^{\infty}$  with $f_{k}$ of the form {\rm (\ref{eq:expsum})}
with the nonzero con-eigenvalues $\sigma_{0} \ge \sigma_{1} \ge \ldots \ge \sigma_{N-1} > 0$.
Further, let $(\sigma_{K}, {\mathbf v}^{(K)})$ be the $K$-th con-eigenpair of ${\mathbf \Gamma}_{{\mathbf f}}$ with $\sigma_{K} \neq  \sigma_{k}$ for $K \neq k$. Then the shift-invariant space
$$ {\cal S}_{\overline{\mathbf v}^{(K)}} := {\rm clos}_{\ell^{2}} {\rm span} \,  \{ S^{l} \overline{\mathbf v}^{(K)} : l \in \N_{0} \} $$
has at least  co-dimension $K$ in $\ell_{2}$, and the matrix ${\mathbf \Gamma}_{{\mathbf f- \mathbf g}^{(K)}}$ with ${\mathbf g}^{(K)}$ determined by {\rm (\ref{g})} has at least rank $K$.
Moreover, for the operator norm of ${\mathbf \Gamma}_{{\mathbf g}^{(K)}}$ we have
$$ \| {\mathbf \Gamma}_{{\mathbf g}^{(K)}} \|_{\ell^{2} \to \ell^{2}} = \| {\mathbf \Gamma}_{{\mathbf f}} - {\mathbf \Gamma}_{{\mathbf f- \mathbf g}^{(K)}} \|_{\ell^{2} \to \ell^{2}} = \sigma_{K}. $$
\end{theorem}

\begin{proof}
1. Similarly as in the proof of  Lemma \ref{lem2} we observe that 
 \begin{eqnarray*}
 ({\mathbf \Gamma}_{\mathbf g^{(K)}} \overline{\mathbf v}^{(K)} )_{l} &=& \sum_{j=0}^{\infty} g_{l+j} \bar{v}_{j}^{(K)} = \sum_{j=0}^{\infty} \frac{\sigma_{K}}{2\pi} \int_{0}^{2\pi} \frac{P_{{\mathbf v}^{(K)}} ({\mathrm e}^{{\mathrm i} t})}{P_{\overline{\mathbf v}^{(K)}}({\mathrm e}^{-{\mathrm i} t})} {\mathrm e}^{-{\mathrm i} t (l+j)} 
 \bar{v}_{j}^{(K)} {\mathrm d} t \\
 &=& \frac{\sigma_{K}}{2\pi} \int_{0}^{2\pi} \frac{P_{{\mathbf v}^{(K)}} ({\mathrm e}^{{\mathrm i} t})}{P_{\overline{\mathbf v}^{(K)}}({\mathrm e}^{-{\mathrm i} t})} {\mathrm e}^{-{\mathrm i} t l}\sum_{j=0}^{\infty} \bar{v}_{j}^{(K)} {\mathrm e}^{-{\mathrm i} t j} {\mathrm d} t \\
&=& \frac{\sigma_{K}}{2\pi} \int_{0}^{2\pi} P_{\mathbf v}^{(K)} ({\mathrm e}^{{\mathrm i} t}) {\mathrm e}^{-{\mathrm i} t l} {\mathrm d} t = \sigma_{K} v_{l}^{(K)},
 \end{eqnarray*}
 for all $k \in \N_{0}$ and thus ${\mathbf \Gamma}_{\mathbf g^{(K)}} \overline{\mathbf v}^{(K)} = \sigma {\mathbf v}^{(K)}$, resp.\ ${\mathbf \Gamma}_{{\mathbf f}-{\mathbf g}^{(K)}} \overline{\mathbf v}^{(K)} = {\mathbf 0}$. 
Moreover, by Lemma \ref{lem2} it follows that $\| {\mathbf \Gamma}_{{\mathbf g}^{(K)}} \|_{\ell^{2} \to \ell^{2}} = \sigma_{K}$. 

\noindent
We consider now the operator ${\mathbf \Gamma}_{{\mathbf f}-{\mathbf g}^{(K)}}$. By Lemma \ref{shiftinv}, the shift-invariant space ${\cal S}_{\overline{\mathbf v}^{(K)}}$  is a subset of ${\Ker} \, {\mathbf \Gamma}_{{\mathbf f}-{\mathbf g}^{(K)}}$. 
On the other hand, we observe that  for $r=0, \ldots , K-1$,
\begin{eqnarray*}
\| {\mathbf \Gamma}_{{\mathbf f}-{\mathbf g}^{(K)}} \overline{\mathbf v}^{(r)} \|_{{2}} &=& 
\| {\mathbf \Gamma}_{\mathbf f} \overline{\mathbf v}^{(r)} - {\mathbf \Gamma}_{{\mathbf g}^{(K)}} \overline{\mathbf v}^{(r)}\|_{{2}} \\
&\ge& | \| {\mathbf \Gamma}_{\mathbf f} \overline{\mathbf v}^{(r)} \|_{{2}} - \| {\mathbf \Gamma}_{{\mathbf g}^{(K)}} \overline{\mathbf v}^{(r)} \|_{{2}} | \ge (\sigma_{r} - \sigma_{k})  \| {\mathbf v}^{(r)} \|_{{2}}>0.
\end{eqnarray*}
Thus, the $K$ linearly independent con-eigenvectors $\overline{\mathbf v}^{(0)}, \ldots , \overline{\mathbf v}^{(K-1)}$ to the larger con-eigen\-val\-ues $\sigma_{0} \ge \ldots \ge \sigma_{k-1}$ are not contained in the kernel of ${\mathbf \Gamma}_{{\mathbf  f}-{\mathbf g}^{(K)}}$ and thus not in ${\cal S}_{\overline{\mathbf v}^{(K)}}$.
Hence, codim ${\cal S}_{\overline{\mathbf v}^{(K)}} \ge  K$, and ${\mathbf \Gamma}_{{\mathbf  f}-{\mathbf g}^{(K)}}$ possesses at least rank $K$.
\end{proof}

Since $P_{\overline{\bf v}^{(K)}}$ has by construction no zeros on the unit circle, it follows that ${\bf g}^{(K)}$ is also in $\ell^{1}$ and therefore also 
$\| {\mathbf \Gamma}_{{\mathbf g}^{(K)}} \|_{\ell^{1} \to \ell^{1}} = \sigma_{K}$.
\medskip

Finally, we conclude the following theorem.

\begin{theorem}\label{AAK1}
Let ${\mathbf \Gamma}_{\mathbf f}$  be the Hankel operator of finite rank $N$ generated by ${\mathbf f} = (f_{k})_{k=0}^{\infty}$ of the form {\rm (\ref{eq:expsum})} and with nonzero singular values $\sigma_{0} \ge \sigma_{1} \ge \ldots 
\ge \sigma_{N-1} > 0$.  Further, let $(\sigma_{K}, {\mathbf v}^{(K)})$ be the $K$-th con-eigenpair of ${\mathbf \Gamma}_{\mathbf f}$.
Then for each $K \in \{ 0, \ldots, N-1\}$ with $\sigma_{K}$ being a single singular value we have:
\begin{itemize}
\item[\textnormal{(1)}] The Laurent polynomial $P_{{\mathbf v}^{(K)}}(z)$ corresponding to the con-eigenvector ${\mathbf v}^{(K)}$ has exactly $K$ zeros 
$z_{1}^{(K)}, \ldots , z_{K}^{(K)}$ in $\D$ repeated according to multiplicity.
\item[\textnormal{(2)}] Considering the Hankel matrix ${\mathbf \Gamma}_{{\mathbf g}^{(K)}}$ given by the  sequence ${\mathbf g}^{(K)}= (g_{k})_{k=0}^{\infty}$ in {\rm (\ref{g})}, it follows that ${\mathbf  \Gamma}_{{\mathbf f} - {\mathbf g}^{(K)}}$ possesses the rank $K$ and 
$$ \| {\mathbf \Gamma}_{{\mathbf g}^{(K)}} \|_{\ell^{p} \to \ell^{p}} = \| {\mathbf \Gamma}_{{\mathbf f}} - {\mathbf  \Gamma}_{{\mathbf f} - {\mathbf g}^{(K)}} \|_{\ell^{p} \to \ell^{p}} = \sigma_{K}, \qquad p \in \{ 1,2 \}.
$$
\item[\textnormal{(3)}] The kernel of ${\mathbf \Gamma}_{{\mathbf f}- {\mathbf g}^{(K)}}$ has co-dimension $K$. If the zeros $z_1^{(K)},\ldots,z_K^{(K)}$ are pairwise different, then it satisfies 
\begin{eqnarray*}
 {\Ker} ({\mathbf \Gamma}_{{\mathbf f}- {\mathbf g}^{(K)}}) &=& {\cal S}_{\overline{\mathbf v}^{(K)}} 
 =  ({\rm clos}_{\ell^{2}} {\rm span} \{ ((z_{1}^{(K)})^{l})_{l=0}^{\infty}, \ldots , ((z_{K}^{(K)})^{l})_{l=0}^{\infty} \})^{\perp},
\end{eqnarray*}
where ${\cal S}_{\overline{\mathbf v}^{(K)}}:={\rm clos}_{\ell^{2}} {\rm span} \,  \{ S^{l} \overline{\mathbf v}^{(K)} : l \in \N_{0} \}$.
\end{itemize}
\end{theorem}

\begin{proof}
1. First we show that ${\cal S}_{\overline{\mathbf v}^{(K)}}  = ({\rm clos}_{\ell^{2}} {\rm span} \{ ((z_{1}^{(K)})^{l})_{l=0}^{\infty}, \ldots , ((z_{n_{K}}^{(K)})^{l})_{l=0}^{\infty} \})^{\perp}$, where $z_{1}^{(K)}, \ldots , z_{n_{K}}^{(K)}$ are all pairwise different zeros of $P_{{\mathbf v}^{(K)}}(z)$  inside $\D$. Indeed for all $l \in {\N}_{0}$,
\begin{eqnarray*}
\langle ((z_{j}^{(K)})^{r})_{r=0}^{\infty}, S^{l} \overline{\mathbf v}^{(K)} \rangle_{\ell^{2}}
&=& \langle (S^{*})^{l} ((z_{j}^{(K)})^{r})_{r=0}^{\infty},  \overline{\mathbf v}^{(K)}\rangle_{\ell^{2}} \\
&=& \sum_{r=0}^{\infty} (z_{j}^{(K)})^{r+l} \, {v}_{r}^{(K)} \\
&=& (z_{j}^{(K)})^{l} \sum_{r=0}^{\infty} (z_{j}^{(K)})^{r} \, {v}_{r}^{(K)} = (z_{j}^{(K)})^{l} \, P_{{\mathbf v}^{(K)}} (z_{j}^{(K)}) =0.
\end{eqnarray*}
Thus, 
$$  {\cal S}_{\overline{\mathbf v}^{(K)}} \perp {\rm span} \{ ((z_{1}^{(K)})^{l})_{l=0}^{\infty}, \ldots , ((z_{n_{K}}^{(K)})^{l})_{l=0}^{\infty} \}). $$

Assume now, that  ${\mathbf u} \in \ell^{2}({\N}_{0})$ satisfies ${\mathbf u} \perp 
{\rm span} \{ ((z_{1}^{(K)})^{l})_{l=0}^{\infty}, \ldots , ((z_{n_{K}}^{(K)})^{l})_{l=0}^{\infty} \})$, i.e., that $P_{\overline{\mathbf u}}(z_{j}^{(K)}) =0$ for $j=1, \ldots , n_{K}$. We show that ${\mathbf u} \in {\cal S}_{\overline{\mathbf v}^{(K)}}$. 
We can rewrite 
$$ P_{\overline{\mathbf u}}(e^{i\omega}) = \prod_{j=1}^{n_{K}} \frac{(e^{i\omega} - z_{j}^{(K)} )}{(1 - \overline{z}_{j}^{(K)} e^{i\omega})} \, P_{{\mathbf w} }(e^{i\omega}) = B^{(K)} (e^{i\omega}) P_{{\mathbf w} }(e^{i\omega})$$
with the same Blaschke product as in (\ref{blal}), where $P_{\mathbf w}(e^{i\omega})$ still corresponds to a sequence ${\mathbf w} = (w_{l})_{l=0}^{\infty} \in \ell^{1}({\N}_{0})$. Equivalently, we have $\overline{\mathbf u} = {\mathbf T}_{{\mathbf b}^{(K)}}
{\mathbf w}$. 
Since ${\mathbf T}_{\overline{\mathbf v}^{(K)}}$ contains the columns $\overline{\mathbf v}^{(K)}, \, S\overline{\mathbf v}^{(K)}, \ldots $, the assertion ${\mathbf u} \in {\cal S}_{\overline{\mathbf v}^{(K)}}$ is equivalent to the assertion that there exists a sequence ${\mathbf y} \in \ell^{2}({\N}_{0})$, such that 
$$ \overline{\mathbf u} = {\mathbf T}_{{\mathbf v}^{(K)}} {\mathbf y}. $$
By Lemma \ref{TB} and (\ref{toep1}) this is equivalent to 
$$ {\mathbf T}_{{\mathbf b}^{(K)}} {\mathbf w} = {\mathbf T}_{{\mathbf b}^{(K)}} {\mathbf T}_{{\mathbf u}^{(K)}} \, {\mathbf y},$$ and thus to
$$ {\mathbf w} = {\mathbf T}_{{\mathbf b}^{(K)}}^{*} {\mathbf T}_{{\mathbf b}^{(K)}}
 {\mathbf w} = {\mathbf T}_{{\mathbf u}^{(K)}} \, {\mathbf y}. $$
 The assertion now follows since ${\mathbf T}_{{\mathbf u}^{(K)}}$ is invertible.
 Indeed, (\ref{u}) implies 
 $$ P_{{\mathbf u}^{(K)}}(z) = \frac{1}{\sigma_{K}} \frac{\prod\limits_{j=1}^{n_{K}} (1- \overline{z}_{j}^{(K)} z) \prod\limits_{j=n_{K}+1}^{N-1} (-\beta_{j}^{(K)}) (1 - (\beta_{j}^{(K)})^{(-1)} z)}{\prod\limits_{j=1}^{N}(1- z_{j} z)}, $$
 and thus
 $$ {\mathbf T}_{{\mathbf u}^{(K)}}^{-1} = \frac{\sigma_{K}}{\prod_{j=n_{K}+1}^{N-1} (-\beta_{j}^{(K)})} \, \left( \prod_{j=1}^{n_{K}} {\mathbf T}_{\overline{z}_{j}^{(K)}} \right)
 \left(\prod_{j=n_{K}+1}^{N-1} {\mathbf T}_{({\beta}_{j}^{(K)})^{-1}} \right) \, {\mathbf T}_{\tilde{\mathbf  p}},
 $$
 where ${\mathbf T}_{\overline{z}_{j}^{(K)}}$, ${\mathbf T}_{({\beta}_{j}^{(K)})^{-1}}$ and ${\mathbf T}_{\tilde{\mathbf  p}}$ are the infinite Toeplitz matrices  generated by the sequences $((\overline{z}_{j}^{(K)})^{r})_{r=0}^{\infty}$, $(({\beta}_{j}^{(K)})^{-r})_{r=0}^{\infty}$ and by the finite sequence 
 $\tilde{\mathbf p} = (1, p_{N-1}, \ldots , p_{0})$ containing the coefficients of the Prony polynomial in (\ref{prony}).
 
 2. By Lemma \ref{aak2} we have $n_{K} \le K$, i.e., ${\cal S}_{\overline{\mathbf v}^{(K)}}$ possesses at most  co-dimension $K$. On the other hand, ${\cal S}_{\overline{\mathbf v}^{(K)}} \subseteq \Ker ({\mathbf \Gamma}_{{\mathbf f} - {\mathbf g}^{(K)}})$ and $\Ker ({\mathbf \Gamma}_{{\mathbf f} - {\mathbf g}^{(K)}})$  has at least co-dimension $K$ by Theorem \ref{theo}. Thus, $n_{K} = K$, i.e., $P_{{\mathbf v}^{(K)}}(z)$ possesses exactly $K$ zeros in $\D$, and ${\cal S}_{\overline{\mathbf v}^{(K)}} = \Ker ({\mathbf \Gamma}_{{\mathbf f} - {\mathbf g}^{(K)}})$ has co-dimension $K$. Assertion (2) follows directly from Theorem \ref{theo}.
 \end{proof}
\medskip

{\bf Proof of Theorem \ref{AAK}}.
Theorem \ref{AAK} is now a corollary of Theorem \ref{AAK1}.
Theorem \ref{AAK1} contains the explicit sequence $\tilde{\mathbf f} = {\mathbf g} - {\mathbf f}$. From Theorem \ref{AAK1}(3) it follows that $\tilde{\mathbf f} \in {\rm clos}_{\ell^{2}} {\rm span} \{ ((z_{1}^{(K)})^{l})_{l=0}^{\infty}, \ldots , ((z_{K}^{(K)})^{l})_{l=0}^{\infty} \}$,
i.e., it can be written as a finite linear combination of the form (\ref{eq:AAK1}). 
Moreover, for ${\bf e}_{0}:=(1,0,0, \ldots) \in \ell^{1} \subset \ell^{2}$ we have for $p\in \{ 1, 2 \}$
$$ \| {\bf f} \|_{p} = \frac{\| {\bf \Gamma}_{\bf f} {\bf e}_{0} \|_{p}}{\| {\bf e}_{0} \|_{p}} \le \| {\bf \Gamma}_{\bf f} \|_{\ell^{p} \to \ell^{p}}. $$
Thus the assertion follows. \hfill $\Box$
\bigskip

\noindent
{\bf Connection to Prony's method.}
There is now obviously a close connection to Prony's method.
When taking a zero-(con)-eigenvector ${\mathbf v}$ of the Hankel operator ${\mathbf \Gamma}_{\mathbf f}$ generated by ${\mathbf  f}$ in (\ref{eq:expsum}), then by Theorem \ref{zerovectors}, the Laurent polynomial corresponding to ${\mathbf v}$ satisfies $P_{\overline{\mathbf v}}(z_{j}) =0$ for all $j=1, \ldots , N$ and has therefore at least $N$ zeros inside the unit disk $\D$. In particular we have

\begin{corollary} \label{pronymethod}
Let ${\mathbf v} = {\mathbf v}^{(N)} \in \ell^{1}$ be a zero-(con)-eigenvector of ${\mathbf \Gamma}_{{\mathbf f}}$ with ${\mathbf f} $ in (\ref{eq:expsum}), such that $$P_{\overline{\mathbf v}}(z) =  \prod_{j=1}^{N} (z - z_{j}) = P(z)$$
is the Prony polynomial in (\ref{prony}).
Then 
$${\Ker} ({\mathbf \Gamma}_{{\mathbf f}}) = {\cal S}_{\overline{\mathbf v}} = {\rm clos}_{\ell^{2}} {\rm span} \,  \{ S^{l} \overline{\mathbf v}^{(N)} : l \in \N_{0} \} 
 =  ({\rm clos}_{\ell^{2}} {\rm span} \{ (z_{1}^{l})_{l=0}^{\infty}, \ldots , (z_{N}^{l})_{l=0}^{\infty} \})^{\perp}.$$
\end{corollary}

\begin{proof} 
We observe as before that 
${\Ker} ({\mathbf \Gamma}_{{\mathbf f}}) \perp 
{\rm clos}_{\ell^{2}} {\rm span} \{ (z_{1}^{l})_{l=0}^{\infty}, \ldots , (z_{N}^{l})_{l=0}^{\infty} \}$. Since on the one hand 
${\cal S}_{\overline{\mathbf v}} \subseteq {\Ker} ({\mathbf \Gamma}_{{\mathbf f}})$ by Lemma \ref{shiftinv}, and both ${\Ker} ({\mathbf \Gamma}_{{\mathbf f}})$ and ${\cal S}_{\overline{\mathbf v}}$ have co-dimension $N$, 
equality follows.
\end{proof}

\begin{remark}
The proof given in this subsection does not explicitly use the Theorems of Beurling and Nehari for Hankel operators.
Nehari's result states that the norm of the operator ${\bf \Gamma}_{\bf f}$ is equal to the infimum of the $L^{\infty}$-norm over all bounded function $2\pi$-periodic functions whose Fourier coefficients coincide with $f_{k}$ for $k\in {\N}_{0}$, see e.g. \cite{Young}. This result is ``hidden'' in Lemma \ref{lem2}, where a sequence ${\bf w}$ is constructed by the Fourier coefficients of a special function with norm $1$ in $L^{\infty}$.

Beurling's theorem essentially says that the linear span of all shifts of a given sequence ${\bf v}$ in $\ell^{2}$ is characterized by the inner factor of its 
corresponding Laurent polynomial $P_{{\bf v}}(z)$. Thus assertion (3) of Theorem \ref{AAK1} is a direct consequence of Beurling's theorem. We have proven it directly by showing invertibility of the Toeplitz matrix ${\bf T}_{{\bf u}^{(K)}}$.
\end{remark}


\section{Numerical Examples}
\label{sec4} 
\setcounter{equation}{0}

In this section we present some numerical examples demonstrating the performance of our algorithm. In all examples the approximate Prony method APM2 provided in \cite{APM} was used in the first step of the algorithm in Section $3$.
\medskip

\noindent
{\bf Example 1.} We approximate the function $f:\R^+ \to \C$ of the form
$$
f(x) = \sum_{j=1}^N a_j\, z_{j}^{x}\,, 
$$
with $N=10$ using $M=50$ samples. We denote by ${\bf f}:=(f_k)_{k=0}^M:=(f(k))_{k=0}^M$ the vector of samples of the function $f$ and by ${\bf \tilde{f}}^{(n)}:=(\tilde{f}^{(n)}_k)_{k=0}^M$ the output vector of the $n$-term approximation $\tilde{f}$ of $f$. Both, the nodes $z_j$ and the coefficients $a_j,~j=1,\ldots,10$ were chosen randomly in $\D$ and in the interval $(0,1)$ respectively and are given as follows
$$
\begin{pmatrix}
z_1 \\
z_2 \\ 
z_3 \\ 
z_4 \\ 
z_5 \\ 
z_6 \\ 
z_7 \\ 
z_8 \\ 
z_9 \\ 
z_{10} \end{pmatrix} =
\left(\begin{array}{r}
  -0.0159 + 0.3739i\\
  -0.0770 + 0.0394i\\
  -0.0639 - 0.1791i\\
  -0.2324 + 0.5268i\\
   0.0102 + 0.4511i\\
   0.0129 + 0.0602i\\
   0.3812 + 0.1470i\\
   0.3538 + 0.1045i\\
   0.1732 - 0.3507i\\
  -0.1457 - 0.2385i\end{array}\right), \qquad 
\begin{pmatrix}
a_1 \\
a_2 \\ 
a_3 \\ 
a_4 \\ 
a_5 \\ 
a_6 \\ 
a_7 \\ 
a_8 \\ 
a_9 \\ 
a_{10} \end{pmatrix} = 
\left(\begin{array}{r}
   0.4709 + 0.4302i\\
   0.2305 + 0.1848i\\
   0.8443 + 0.9049i\\
   0.1948 + 0.9797i\\
   0.2259 + 0.4389i\\
   0.1707 + 0.1111i\\
   0.2277 + 0.2581i\\
   0.4357 + 0.4087i\\
   0.3111 + 0.5949i\\
   0.9234 + 0.2622i \end{array}\right).
$$
\begin{figure}[h]
  \centering
      \includegraphics[width=0.4\textwidth]{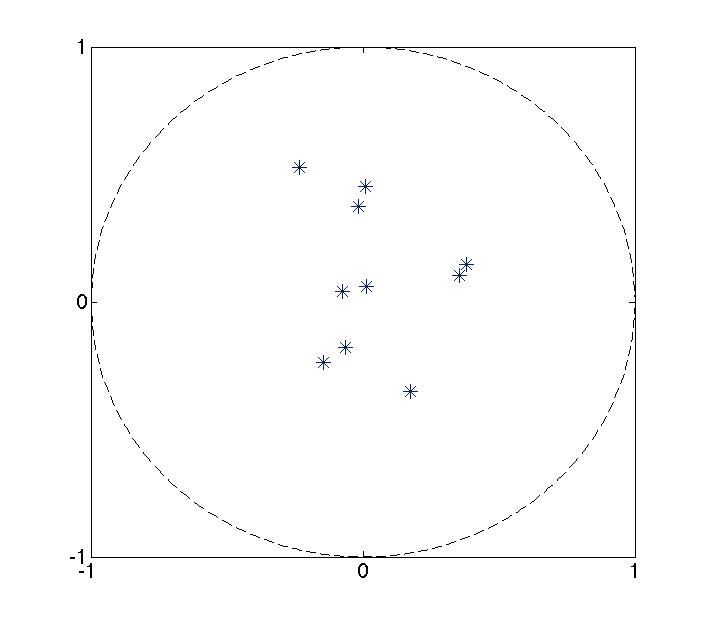}
  \caption{The nodes $z_1,\ldots,z_{10}$ in the unit circle from Example 1.}
\end{figure}
The minimization problem from step $4$ was performed by the least squares method using the given $M$ samples. In the first step of the algorithm the accuracies $\varepsilon_1 =\varepsilon_2 =10^{-15}$, the radius $r =1$ and the upper bound for the number of exponentials $L=20$ were chosen for APM2. The following table shows the coneigenvalues $\sigma_n$ of the matrix ${\bf A}_{N} {\bf Z}_{N}$ and the corresponding approximation errors for different values of $n$.
\renewcommand*\arraystretch{1.2}
$$\setlength\arraycolsep{15pt}
 \begin{array}{|c|c|c|} \hline
 n		& \sigma_n					& \| {\bf f} - {\bf \tilde{f}}^{(n)} \|_{{2}}	\\ \hline
1		&   4.4340\textnormal{e-}01		&  4.4142\textnormal{e-}01		\\
2		&   5.5171\textnormal{e-}02		&  5.3850\textnormal{e-}02		\\
3		&   1.8185\textnormal{e-}02		&  1.8096\textnormal{e-}02		\\
4		&   8.1149\textnormal{e-}03		&  8.1145\textnormal{e-}03		\\
5		&   7.8571\textnormal{e-}05		&  7.8571\textnormal{e-}05		\\
6		&   4.3647\textnormal{e-}06		&  4.3647\textnormal{e-}06		\\
7		&   2.6711\textnormal{e-}07 		&  2.6711\textnormal{e-}07		\\
8		&   6.2531\textnormal{e-}08 		&  6.2531\textnormal{e-}08		\\
9 		&   1.4512\textnormal{e-}10 		&  1.4512\textnormal{e-}10		\\ \hline
 \end{array}
$$
\renewcommand*\arraystretch{1}

\noindent
{\bf Example 2.} In this example we approximate the function $f(x) = 1/x$ using $M=100$ samples in the interval $[1,50]$. Let ${\bf f}:=(f_k)_{k=0}^M$ be the vector of samples of $f$ and ${\bf \tilde{f}}^{(n)}:=(\tilde{f}^{(n)}_k)_{k=0}^M$ the output vector of its $n$-term approximation by above algorithm. The initial $N=11$ nodes $z_j$ and weights $a_j$ were obtained with parameters $\varepsilon_1 =\varepsilon_2 =10^{-10}, r =1,1$ and $L=23$ by applying APM2 and are given as follows:
$$
\begin{pmatrix}
z_1 \\
z_2 \\ 
z_3 \\ 
z_4 \\ 
z_5 \\ 
z_6 \\ 
z_7 \\ 
z_8 \\ 
z_9 \\ 
z_{10} \\
z_{11} \end{pmatrix} =
\left(\begin{array}{r}
    0.9959 \\
    0.9781 \\
    0.9443 \\
    0.8919 \\
    0.8178 \\
    0.7198 \\
    0.5981 \\
    0.4568 \\
    0.3060 \\
    0.1634 \\
    0.0533\end{array}\right), \qquad 
\begin{pmatrix}
a_1 \\
a_2 \\ 
a_3 \\ 
a_4 \\ 
a_5 \\ 
a_6 \\ 
a_7 \\ 
a_8 \\ 
a_9 \\ 
a_{10} \\
a_{11} \end{pmatrix} = 
\left(\begin{array}{r}
    0.0214 \\
    0.0507 \\
    0.0818 \\
    0.1137 \\
    0.1422 \\
    0.1597 \\
    0.1585 \\
    0.1339 \\
    0.0895 \\
    0.0405 \\
    0.0082\end{array}\right).
$$
In the following table we compare the $\ell^2$-error of the above algorithm with the $n$-term exponential sum approximation obtained by W. Hackbusch in \cite{HB}. Let ${\bf f}_H^{(n)}$ be a vector of samples of the $n$-term approximation of $f$ from \cite{HB}, then we obtain the following errors using the same sampling for both approximations.
\renewcommand*\arraystretch{1.2}
$$\setlength\arraycolsep{15pt}
 \begin{array}{|c|c|c|c|} \hline
 n		& \sigma_n					& \| {\bf f} - {\bf \tilde{f}}^{(n)} \|_{{2}}	& \| {\bf f} - {\bf {f}}_H^{(n)} \|_{{2}}	\\ \hline
1		&  1.5789\textnormal{e-}00		&  1.0479\textnormal{e-}00		& -						\\
2		&  4.3137 \textnormal{e-}01		&  3.7340\textnormal{e-}01		& 1.4145\textnormal{e-}01		\\
3		&   9.9203\textnormal{e-}02		&  9.4372\textnormal{e-}02		& 2.4771\textnormal{e-}02		\\
4		&   1.9627\textnormal{e-}02		&  1.9207\textnormal{e-}02		& 4.4988\textnormal{e-}03		\\
5		&   3.3233\textnormal{e-}03		&  3.2870\textnormal{e-}03		& 7.8479\textnormal{e-}04		\\
6		&   4.7360\textnormal{e-}04		&  4.6840\textnormal{e-}04		& 1.3138\textnormal{e-}04		\\
7		&   5.5123\textnormal{e-}05 		&  5.4309\textnormal{e-}05		& 2.2138\textnormal{e-}05		\\
8		&   4.9665\textnormal{e-}06 		&  4.8884\textnormal{e-}06		& 3.6552\textnormal{e-}06		\\
9 		&   3.1299\textnormal{e-}07 		&  3.1581\textnormal{e-}07		& 5.9684\textnormal{e-}07		\\
10 		&   1.0840\textnormal{e-}08 		&  4.5328\textnormal{e-}08		& 9.8033\textnormal{e-}08		\\ \hline
 \end{array}
$$
Note that in \cite{HB} an algorithm for function approximation based on the Remez algorithm was used, whereas we approximate sequences of samples. The nodes $\tilde{z}_j$ and weights $\tilde{a}_j$ for $j=1,\ldots,10$ obtained by Hackbusch are very different from the ones computed by our algorithm, especially for small $n$. Even though, for the most values of $n$ we obtain the same order of the $\ell^2$-error. For $n\geq 9$ the error of our algorithm is even slightly better. The nodes and weights obtained by our algorithm are given below.
$$\setlength\arraycolsep{4pt}
 \begin{array}{|c|cccccccccc|} \hline
n			& 1		& 2		& 3		& 4		& 5		& 6		& 7		& 8		& 9		& 10		\\ \hline
\tilde{z}_1		& 0.9804	& 0.8725	& 0.6982	& 0.5254	& 0.3856	& 0.2816 	& 0.2063	& 0.1516	& 0.1112	& 0.0802	\\
\tilde{z}_2		&		& 0.9933  	& 0.9545  & 0.8706   & 0.7544  & 0.6279  & 0.5079  	& 0.4022	& 0.3123  & 0.2355	\\
\tilde{z}_3		&		&		& 0.9953	& 0.9710	& 0.9187	& 0.8386	& 0.7391	& 0.6309	& 0.5220  & 0.4157	\\
\tilde{z}_4		&		&		&		& 0.9958	& 0.9760	& 0.9358	& 0.8731	& 0.7901	& 0.6917  & 0.5814	\\
\tilde{z}_5		&		&		&		&		& 0.9959	& 0.9776	& 0.9421	& 0.8869	& 0.8113  & 0.7154	\\
\tilde{z}_6		&		&		&		&		&		& 0.9959 	& 0.9780	& 0.9439	& 0.8911  & 0.8171	\\
\tilde{z}_7		&		&		&		&		&		&		& 0.9959	& 0.9780	& 0.9443  & 0.8919	\\
\tilde{z}_8		&		&		&		&		&		&		&		& 0.9959	& 0.9781  & 0.9443	\\
\tilde{z}_9		&   		&		&		&		&		&		&		&		& 0.9959	& 0.9781	\\
\tilde{z}_{10}	&		&		&		&		&		&		&		&		& 		& 0.9959	\\ \hline
 \end{array}
$$
$$\setlength\arraycolsep{4pt}
 \begin{array}{|c|cccccccccc|} \hline
n			& 1		& 2		& 3		& 4		& 5		& 6		& 7		& 8		& 9		& 10		\\ \hline
\tilde{a}_1		& 0.2419	& 0.6728	& 0.7437	& 0.5561	& 0.3499	& 0.2045 	& 0.1158	& 0.0645	& 0.0354	& 0.0186	\\
\tilde{a}_2		&		& 0.0434  	& 0.1717  	& 0.3228   & 0.3886  & 0.3592  & 0.2832  	& 0.2011	& 0.1320  & 0.0797	\\
\tilde{a}_3		&		&		& 0.0290	& 0.0921	& 0.1756	& 0.2434	& 0.2669	& 0.2474	& 0.2014  & 0.1455	\\
\tilde{a}_4		&		&		&		& 0.0225	& 0.0637	& 0.1173	& 0.1685	& 0.1988	& 0.2001  & 0.1753	\\
\tilde{a}_5		&		&		&		&		& 0.0217	& 0.0542	& 0.0929	& 0.1318	& 0.1597  & 0.1686	\\
\tilde{a}_6		&		&		&		&		&		& 0.0215 	& 0.0514	& 0.0841	& 0.1172  & 0.1444	\\
\tilde{a}_7		&		&		&		&		&		&		& 0.0214	& 0.0508	& 0.0821  & 0.1140	\\
\tilde{a}_8		&		&		&		&		&		&		&		& 0.0214	& 0.0507  & 0.0818	\\
\tilde{a}_9		&   		&		&		&		&		&		&		&		& 0.0214	& 0.0507	\\
\tilde{a}_{10}	&		&		&		&		&		&		&		&		& 		& 0.0214	\\ \hline
 \end{array}
$$

\section*{Acknowledgement}
The authors gratefully acknowledge the funding of this work by the DFG in the framework of the GRK 2088 and the project 
PL 170/16-1.


\end{document}